\documentclass[12pt,reqno]{amsart}%
\usepackage{amsfonts}
\usepackage{amsmath}
\usepackage{amssymb}
\usepackage{geometry}
\usepackage{graphicx}%
\providecommand{\U}[1]{\protect\rule{.1in}{.1in}}
\newtheorem{theorem}{Theorem}[section]
\theoremstyle{plain}

\newtheorem{lemma}{Lemma}[section]

\newtheorem{remark}{Remark}

\numberwithin{equation}{section}
\begin{document}
\title[Optimizers for Caffarelli-Kohn-Nirenberg inequalities]{Sharp constants and optimizers for a class of the Caffarelli-Kohn-Nirenberg inequalities}
\author{Nguyen Lam}
\address{Department of Mathematics\\
University of Pittsburgh\\
Pittsburgh, PA 15260, USA}
\email{nhlam@pitt.edu}
\author{Guozhen Lu }
\address{Department of Mathematics\\
Wayne State University\\
Detroit, MI 48202, USA}
\email{gzlu@wayne.edu}
\thanks{Research of this work was partly supported by a US NSF grant DMS-1301595. The first author is partly supported by the AMS-Simons travel grant. The
second author is partly supported by a Simons Fellowship from the Simons
Foundation. }
\subjclass{Primary: 26D10; Secondary 46E35. }
\keywords{Caffarelli-Kohn-Nirenberg inequality; Symmetry; Maximizers; Sharp constants.}

\begin{abstract}
In this paper, we will use a suitable tranform to investigate the sharp
constants and optimizers for the following Caffarelli-Kohn-Nirenberg
inequalities for a wide range of parameters $(r,p,q,s,\mu,\sigma)$ and $0\leq
a\leq1$:
\begin{equation}
\left(
{\displaystyle\int}
\left\vert u\right\vert ^{r}\frac{dx}{\left\vert x\right\vert ^{s}}\right)
^{1/r}\leq C\left(
{\displaystyle\int}
\left\vert \nabla u\right\vert ^{p}\frac{dx}{\left\vert x\right\vert ^{\mu}%
}\right)  ^{a/p}\left(
{\displaystyle\int}
\left\vert u\right\vert ^{q}\frac{dx}{\left\vert x\right\vert ^{\sigma}%
}\right)  ^{\left(  1-a\right)  /q}. \tag{CKN}%
\end{equation}
We are able to compute the best constants and the explicit forms of the
extremal functions in numerous cases. When $0<a<1$, we can deduce the
existence and symmetry of optimizers for a wide range of parameters. Moreover,
in the particular classes $r=p\frac{q-1}{p-1}$ and $q=p\frac{r-1}{p-1}$, the
forms of maximizers will also be provided in the spirit of Del Pino and
Dolbeault \cite{DePDo1, DePDo}. In the case $a=1$, that is the
Caffarelli-Kohn-Nirenberg inequality without the interpolation term, we will
provide the exact maximizers for all the range of $\mu\geq0$. The
Caffarelli-Kohn-Nirenberg inequalities with arbitrary norms on the Euclidean
spaces will also be considered in the spirit of Cordero-Erausquin, Nazaret and
Villani \cite{CeNV}.

\end{abstract}
\maketitle

\section{Introduction}

Geometric and Functional inequalities have a wide range of applications and
play a crucial role in geometric analysis, partial differential equations and
other branches of modern mathematics. In many situations, the validity of the
inequality and some explicit bounds for its best constant are enough to run
the process. However, there are numerous circumstances that it is required to
know the exact sharp constant and information on extremal functions.

Among those inequalities, the Caffarelli-Kohn-Nirenberg (CKN) inequality is
one of the most important and interesting ones. It is worth noting that many
well-known and important inequalities such as Gagliardo-Nirenberg
inequalities, Sobolev inequalities, Hardy-Sobolev inequalities, Nash's
inequalities, etc are just the special cases of the CKN inequalities.

The CKN inequalities were first introduced in 1984 by Caffarelli, Kohn and
Nirenberg in their celebrated work \cite{CKN}:\vskip0.5cm

\textbf{Theorem A.} \textit{There exists a positive constant }$C=C\left(
N,r,p,q,\gamma,\alpha,\beta\right)  $\textit{ such that for all }$u\in
C_{0}^{\infty}\left(
\mathbb{R}
^{N}\right)  :$%
\begin{equation}
\left\Vert \left\vert x\right\vert ^{\gamma}u\right\Vert _{r}\leq C\left\Vert
\left\vert x\right\vert ^{\alpha}\left\vert \nabla u\right\vert \right\Vert
_{p}^{a}\left\Vert \left\vert x\right\vert ^{\beta}u\right\Vert _{q}^{1-a}
\label{*}%
\end{equation}
\textit{where}%
\begin{align*}
p,q  &  \geq1,~r>0,~0\leq a\leq1\\
\frac{1}{p}+\frac{\alpha}{N},~\frac{1}{q}+\frac{\beta}{N},~\frac{1}{r}%
+\frac{\gamma}{N}  &  >0\text{ where}\\
\gamma &  =a\sigma+\left(  1-a\right)  \beta\\
\frac{1}{r}+\frac{\gamma}{N}  &  =a\left(  \frac{1}{p}+\frac{\alpha-1}%
{N}\right)  +(1-a)\left(  \frac{1}{q}+\frac{\beta}{N}\right)  ,
\end{align*}
\textit{and}%
\begin{align*}
0  &  \leq\alpha-\sigma\text{ if }a>0\text{ and}\\
\alpha-\sigma &  \leq1\text{ if }a>0\text{ and }\frac{1}{p}+\frac{\alpha-1}%
{N}=\frac{1}{r}+\frac{\gamma}{N}.
\end{align*}

\vskip0.5cm

In this paper, if we perform the following change of exponents as in
\cite{ZhZ}:%
\[
\alpha=-\frac{\mu}{p},~\beta=-\frac{\theta}{q},~\gamma=-\frac{s}{r}.
\]
Then (\ref{*}) will become the following equivalent form:%
\begin{equation}
\left(  \int\limits_{%
\mathbb{R}
^{N}}\left\vert u\right\vert ^{r}\frac{dx}{\left\vert x\right\vert ^{s}%
}\right)  ^{1/r}\leq C\left(  \int\limits_{%
\mathbb{R}
^{N}}\left\vert \nabla u\right\vert ^{p}\frac{dx}{\left\vert x\right\vert
^{\mu}}\right)  ^{a/p}\left(  \int\limits_{%
\mathbb{R}
^{N}}\left\vert u\right\vert ^{q}\frac{dx}{\left\vert x\right\vert ^{\theta}%
}\right)  ^{\left(  1-a\right)  /q} \tag{CKN}%
\end{equation}
where
\[
a=\frac{\left[  \left(  N-\theta\right)  r-\left(  N-s\right)  q\right]
p}{\left[  \left(  N-\theta\right)  p-\left(  N-\mu-p\right)  q\right]  r}.
\]
We will restrict our consideration in this paper to the case $1<p<N.$

\vskip0.5cm

When $s=\mu=\theta=0$ and $a=1$, we recover the well-known Sobolev inequality:%
\begin{equation}
\left(  \int\limits_{%
\mathbb{R}
^{N}}\left\vert u\right\vert ^{p^{\ast}}dx\right)  ^{1/p^{\ast}}\leq S\left(
N,p\right)  \left(  \int\limits_{%
\mathbb{R}
^{N}}\left\vert \nabla u\right\vert ^{p}dx\right)  ^{1/p} \label{S}%
\end{equation}
where $p^{\ast}=\frac{Np}{N-p}$. This inequality has important applications in
many areas of mathematics and there is a vast literature. We just mention here
that when $p>1$, the best constant $S\left(  N,p\right)  $ was found in the
works of Aubin \cite{Aubin} and Talenti \cite{Talenti} using rather classical
tools such as Schwarz rearrangement, and solution of a particular
one-dimensional problem, the Bliss inequality. The case $p=2$ was explored
more by Beckner in \cite{Beck} due to its conformal invariance. For $p=1$, it
has been known that the Sobolev inequality is equivalent to the classical
Euclidean isoperimetric inequality.

\smallskip

When $a=1$, $\mu=0$, $0\leq s\leq p<N$ and $r=p^{\ast}\left(  s\right)
=\frac{N-s}{N-p}p,$ the CKN inequality becomes the Hardy-Sobolev (HS)
inequality that is the interpolation of the Sobolev inequality and the Hardy
inequality:%
\begin{equation}
\left(  \int\limits_{%
\mathbb{R}
^{N}}\left\vert u\right\vert ^{p^{\ast}\left(  s\right)  }\frac{dx}{\left\vert
x\right\vert ^{s}}\right)  ^{1/p^{\ast}\left(  s\right)  }\leq HS\left(
N,p,s\right)  \left(  \int\limits_{%
\mathbb{R}
^{N}}\left\vert \nabla u\right\vert ^{p}dx\right)  ^{1/p}. \label{HS}%
\end{equation}
In this situation, in \cite{Lieb1} Lieb applied the symmetrization arguments
to study (\ref{HS}) in the case $p=2$ and gave the best constants and explicit
optimizers. The study of the best constant $HS\left(  N,p,s\right)  $ and
extremal functions for the inequalities (\ref{HS}) in the general range goes
back to Ghoussoub and Yuan in \cite{GY} and maybe even earlier (see references
in \cite{GY}): The maximizers for the HS inequality when $0\leq s<p<N$ are the
functions
\begin{equation}
u_{c,\lambda}\left(  x\right)  =c\left(  \lambda+\left\vert x\right\vert
^{\frac{p-s}{p-1}}\right)  ^{-\frac{N-p}{p-s}}\text{ for some }c\neq0\text{,
}\lambda>0\text{.} \label{HSE}%
\end{equation}
Actually, $u_{c,\lambda}$ (after rescaling) is the only positive radial
solutions of
\[
-\operatorname{div}\left(  \left\vert \nabla u\right\vert ^{p-2}\nabla
u\right)  =\frac{u^{p^{\ast}\left(  s\right)  -1}}{\left\vert x\right\vert
^{s}}\text{ on }%
\mathbb{R}
^{N}.
\]

When $a=1$ and $0<\mu,s<N$, the CKN inequality does not contain the
interpolation term. There are great efforts to investigate the sharp
constants, existence/nonexistence and symmetry/symmetry breaking of maximizers
in this situation, especially when $p=2$. See \cite{BM, CW, ChCh, DePDoFT,
FMT, WW}, among others. For instance, Chou and Chu considered the case $p=2$
and $\frac{\mu}{2}\leq\frac{s}{r}\leq\frac{\mu}{2}+1$ and provided the best
constants and explicit optimizers. In \cite{WW}, Wang and Willem studied the
compactness of all maximizing sequences up to dilations in the spirit of Lions
\cite{Lions1, Lions2, Lions3, Lions4}. In \cite{CW}, Catrina and Wang
investigated the class of $p=2$ and $\mu<0$ and established the
attainability/inattainability and symmetry breaking of extremal functions.
Caldiroli and Musina studied the symmetry breaking of extremals for the CKN
inequalities in a non-Hilbertian setting in \cite{CM}. In a recent paper
\cite{DEL}, Dolbeault, Esteban and Loss studied the characterization of the
optimal symmetry breaking region in HS inequalities with $p=2$. As a
consequence, maximizers and best constants are calculated in the symmetry
region. Their result solves a longstanding conjecture on the optimal symmetry range.

\smallskip

In the case $0<a<1$, the CKN inequality includes the interpolation term. This
situation is much harder to study. When there is no singular term, i.e.
$s=\theta=\mu=0$, the non-weighted CKN inequality, namely the
Gagliardo-Nirenberg inequality, has been studied at length by many authors,
see e.g., \cite{BFV, DELL, DFLP, HM}, to mention just a few. Especially, for
very particular classes, the best constant and the maximizers for the
Gagliardo-Nirenberg inequality are provided explicitly by Del Pino and
Dolbeault in \cite{DePDo1, DePDo}. Indeed, in the special class $r=p\frac
{q-1}{p-1}$, Del Pino and Dolbeault proved that the maximizers for the
Gagliardo-Nirenberg inequality have the form $A\left(  1+B\left\vert
x-\overline{x}\right\vert ^{\frac{p}{p-1}}\right)  ^{-\frac{p-1}{q-p}}$ while
in the case $q=p\frac{r-1}{p-1}$, the optimizers are $A\left(  1-B\left\vert
x-\overline{x}\right\vert ^{\frac{p}{p-1}}\right)  _{+}^{-\frac{p-1}{r-p}}$,
for some $A\in%
\mathbb{R}
$, $B>0$ and $\overline{x}\in%
\mathbb{R}
^{N}$. See also \cite{Agu1, Agu2} where Agueh gives a proof by studying a
$p-$Laplacian type equation and by transforming the unknown of the equation
via some change of functions. We also cite \cite{CeNV} where
Cordero-Erausquin, Nazaret, and Villani set up a beautiful link between
optimal transportation and certain Sobolev inequalities and
Gagliardo-Nirenberg inequalities.

\smallskip

However, as far as we know, there are only a few papers concerning the full
weighted CKN inequalities (i.e., $0<a<1$ and at least one of $s, \mu, \theta$
is nonzero), see the review paper by Dolbeault and Esteban \cite{DE0}.
Compared with the special cases of the Gagliardo-Nirenberg inequalities
without the interpolation term (i.e., $a=1$), dealing with such CKN
inequalities encounters considerably more difficulty. For instance, the
Fourier analysis techniques cannot be applied in this setting. Moreover, the
classical Schwarz rearrangement, which is based on an isoperimetric
inequality, is unavailable due to the presence of singular terms (i.e., the
weights $\frac{1}{|x|^{s}}$, $\frac{1}{|x|^{\theta}}$ and $\frac{1}{|x|^{\mu}%
}$). It is worth noting that symmetrization has been a very useful and
efficient (and almost inevitable) method when dealing with the sharp geometric
inequalities. Hence in general we are not able to reduce to work our problem
on CKN inequalities to the case of radial setting. Actually, the problem of
symmetry and symmetry breaking of optimizers for the CKN inequalities has been
investigated by many researchers, see \cite{DE0, DE2, DETT} for instance.

\smallskip

In the particular case $p=q=2$, $-\frac{N-2}{2}<\alpha$, $\beta=\alpha-1$,
$\alpha-1\leq\gamma<\alpha$, $r=\frac{2N}{N+2\left(  \gamma-\alpha\right)  }$,
together with other conditions, the best constant and maximizers for CKN
inequality (\ref{*}) in Theorem A together with the weighted logarithmic
inequality are investigated by Dolbeault and Esteban \cite{DE1}, Dolbeault,
Esteban, Tarantello and Tertikas \cite{DETT}. The authors can also give the
exact sharp constants and the form of maximizers, however, only in the radial setting.

\smallskip

Concerning the inequality (CKN), for the special class $q=\frac{p\left(
r-1\right)  }{p-1}$, $1<p<r$, $N-\theta<\left(  1+\frac{\mu}{p}-\frac{\theta
}{p}\right)  \frac{\left(  r-1\right)  p}{r-p}$ and%
\[
s=\frac{\mu}{p}+1+\frac{p-1}{p}\theta,
\]
Xia could guess and then verify in \cite{Xia} that $\left(  \lambda+\left\vert
x\right\vert ^{1+\frac{\mu}{p}-\frac{\theta}{p}}\right)  ^{-\frac{p-1}{r-p}}$,
$\lambda>0$, are extremal functions. But he could not prove that these are the
all possible optimizers. Moreover, this case does not cover the interesting
situations in \cite{DePDo1, DePDo}.

\smallskip

In \cite{ZhZ}, Zhong and Zou study the existence of extremal functions for the
CKN inequality under a wider region, and use it to set up the continuity and
compactness of embeddings on weighted Sobolev spaces. However, there is no
information about the maximizers provided there.

In a very recent paper \cite{DMN}, the authors studied the CKN inequality in
the regime $s=\theta>0$, $p=2$ and $r=2\left(  q-1\right)  >2$. In this case,
they were able to show that for $s=\theta>0$ small enough, then the CKN
inequality can be achieved by the optimizers of the form $\left(  1+\left\vert
x\right\vert ^{2-s}\right)  ^{-\frac{1}{q-2}}$, up to multiplications by a
constant and scalings.

In \cite{DoLaLu, DoLu, Lam2}, when dealing with the sharp singular
Trudinger-Moser inequalities, which can be considered as the limiting Sobolev
embeddings, where again the classical Schwarz rearrangement could not be used,
the authors propose a new approach. Namely, we define a new Kelvin-type
transform to convert those sharp singular inequalities to the nonweighted
ones. Moreover, in \cite{DoLaLu}, we treat successfully the CKN inequalities
in the special case $p=N,$ $\mu=0,$ $0\leq s=\theta<N$, $1\leq q<r$ and
$a=1-\frac{q}{r}$ using this new transform. Especially, for a 1-parameter
family of inequalities, the best constants and the maximizers for the CKN
inequality are calculated explicitly there.

\vskip0.5cm

Motivated by results in \cite{DoLaLu} and \cite{Agu1, Agu2, CeNV, DePDo, DMN},
in this paper, we will use convenient vector fields to investigate the CKN
inequality in some special regions. Our main idea is that under our suitable
transforms, the CKN inequalities can be converted to the simpler versions,
namely, the Hardy-Sobolev inequalities and Gagliardo-Nirenberg inequalities.
Since the sharp constants and optimizers of those inequalities are easier to
study, and are known in some particular classes, we can get the best constants
and maximizers for CKN inequalities in the corresponding regions.

\smallskip

More precisely, we study the extremal functions for CKN inequality involving
the interpolation term (i..e., $0<a<1$). We will consider the following class:%

\begin{align}
1  &  <p<\text{ }p+\mu<N\text{, }\theta\leq\frac{N\mu}{N-p}\leq s<N,
\tag{C1}\label{C1}\\
1  &  \leq q<r<\frac{Np}{N-p};\text{ }a=\frac{\left[  \left(  N-\theta\right)
r-\left(  N-s\right)  q\right]  p}{\left[  \left(  N-\theta\right)  p-\left(
N-\mu-p\right)  q\right]  r}.\nonumber
\end{align}

Denote $D_{\mu,\theta}^{p,q}\left(
\mathbb{R}
^{N}\right)  $ the completion of the space of smooth compactly supported
functions with the norm $\left(  \int\limits_{%
\mathbb{R}
^{N}}\left\vert \nabla u\right\vert ^{p}\frac{dx}{\left\vert x\right\vert
^{\mu}}\right)  ^{1/p}+\left(  \int\limits_{%
\mathbb{R}
^{N}}\left\vert u\right\vert ^{q}\frac{dx}{\left\vert x\right\vert ^{\theta}%
}\right)  ^{1/q}$, and set
\[
CKN\left(  N,\mu,\theta, s, p, q, r\right)  =\sup_{u\in D_{\mu,\theta}%
^{p,q}\left(
\mathbb{R}
^{N}\right)  }\frac{\left(  \int\limits_{%
\mathbb{R}
^{N}}\left\vert u\right\vert ^{r}\frac{dx}{\left\vert x\right\vert ^{s}%
}\right)  ^{1/r}}{\left(  \int\limits_{%
\mathbb{R}
^{N}}\left\vert \nabla u\right\vert ^{p}\frac{dx}{\left\vert x\right\vert
^{\mu}}\right)  ^{\frac{a}{p}}\left(  \int\limits_{%
\mathbb{R}
^{N}}\left\vert u\right\vert ^{q}\frac{dx}{\left\vert x\right\vert ^{\theta}%
}\right)  ^{\frac{1-a}{q}}}.
\]
Then we have the following result:

\begin{theorem}
\label{T4}Assume (C1). Then $CKN\left(  N,\mu,\theta,s,p,q,r\right)  $ can be
achieved. Moreover, all the extremal functions of $CKN\left(  N,\mu
,\theta,s,p,q,r\right)  $ are radially symmetric.
\end{theorem}

Moreover, we will also give the explicit forms for all maximizers and the
exact best constant for $CKN\left(  N,\mu,\theta,s,p,q,r\right)  $ in the
following special cases:

\begin{theorem}
\label{T5}Assume (C1) with $\theta=s=\frac{N\mu}{N-p}$. If $r=p\frac{q-1}%
{p-1}$, then with $\delta=Np-q\left(  N-p\right)  :$
\begin{align*}
&  CKN\left(  N,\mu,\theta,s,p,q,r\right) \\
&  =\left(  \frac{N-p}{N-p-\mu}\right)  ^{\frac{1}{r}+\frac{p-1}{p}-\frac
{1-a}{q}-\frac{p-1}{p}\left(  1-a\right)  }\\
&  \times\left(  \frac{q-p}{p\sqrt{\pi}}\right)  ^{a}\left(  \frac
{pq}{N\left(  q-p\right)  }\right)  ^{\frac{a}{p}}\left(  \frac{\delta}%
{pq}\right)  ^{\frac{1}{r}}\left(  \frac{\Gamma\left(  q\frac{p-1}%
{q-p}\right)  \Gamma\left(  \frac{N}{2}+1\right)  }{\Gamma\left(  \frac
{p-1}{p}\frac{\delta}{q-p}\right)  \Gamma\left(  N\frac{p-1}{p}+1\right)
}\right)  ^{\frac{a}{N}}%
\end{align*}
and all the maximizers have the form%
\[
V_{0}\left(  x\right)  =A\left(  1+B\left\vert x\right\vert ^{\frac{N-p-\mu
}{N-p}\frac{p}{p-1}}\right)  ^{-\frac{p-1}{q-p}}\text{ for some }A\in%
\mathbb{R}
\text{, }B>0\text{.}%
\]

\end{theorem}

\begin{theorem}
\label{T6}Assume (C1) with $\theta=s=\frac{N\mu}{N-p}$. If $q=p\frac{r-1}%
{p-1}$, then with $\delta=Np-r\left(  N-p\right)  :$
\begin{align*}
&  CKN\left(  N,\mu,\theta,s,p,q,r\right)  \\
&  =\left(  \frac{N-p}{N-p-\mu}\right)  ^{\frac{1}{r}+\frac{p-1}{p}-\frac
{1-a}{q}-\frac{p-1}{p}\left(  1-a\right)  }\\
&  \times\left(  \frac{p-r}{p\sqrt{\pi}}\right)  ^{a}\left(  \frac
{pr}{N\left(  p-r\right)  }\right)  ^{\frac{a}{p}}\left(  \frac{pr}{\delta
}\right)  ^{\frac{1-a}{q}}\left(  \frac{\Gamma\left(  \frac{p-1}{p}%
\frac{\delta}{p-r}+1\right)  \Gamma\left(  \frac{N}{2}+1\right)  }%
{\Gamma\left(  r\frac{p-1}{p-r}+1\right)  \Gamma\left(  N\frac{p-1}%
{p}+1\right)  }\right)  ^{\frac{a}{N}}.
\end{align*}
If $r>2-\frac{1}{p}$, then all the maximizers have the form%
\[
V_{0}\left(  x\right)  =A\left(  1-B\left\vert x\right\vert ^{\frac{N-p-\mu
}{N-p}\frac{p}{p-1}}\right)  _{+}^{-\frac{p-1}{r-p}}\text{ for some }A\in%
\mathbb{R}
\text{, }B>0\text{.}%
\]

\end{theorem}

We also provide the explicit optimizers for the CKN inequalities in the
following regime:%
\begin{align}
p &  =2<2+\mu<N\text{, }2<r=2(q-1)<\frac{2N}{N-2}\tag{C2}\label{C4}\\
\mu+2 &  >s=\theta>\frac{N\mu}{N-2}\text{; }a=\frac{\left(  N-s\right)
\left[  q-2\right]  }{\left[  \left(  N-s\right)  2-\left(  N-\mu-2\right)
q\right]  (q-1)}\nonumber
\end{align}
Again, we denote%
\[
CKN\left(  N,\mu,s,q\right)  =\sup_{u\in D_{\mu,s}^{2,q}\left(
\mathbb{R}
^{N}\right)  }\frac{\left(  \int\limits_{%
\mathbb{R}
^{N}}\left\vert u\right\vert ^{2(q-1)}\frac{dx}{\left\vert x\right\vert ^{s}%
}\right)  ^{\frac{1}{2(q-1)}}}{\left(  \int\limits_{%
\mathbb{R}
^{N}}\left\vert \nabla u\right\vert ^{2}\frac{dx}{\left\vert x\right\vert
^{\mu}}\right)  ^{\frac{a}{2}}\left(  \int\limits_{%
\mathbb{R}
^{N}}\left\vert u\right\vert ^{q}\frac{dx}{\left\vert x\right\vert ^{s}%
}\right)  ^{\frac{1-a}{q}}}.
\]
Then, we will prove that

\begin{theorem}
\label{T11}There exists $s^{\ast}=s^{\ast}\left(  N,q,\mu\right)  \in\left(
0,N-\left(  q-1\right)  \left(  N-2-\mu\right)  \right)  $ such that for all
$\frac{N\mu}{N-2}<s<s^{\ast}$, $CKN\left(  N,\mu,s,q\right)  $ is attained by
the optimizers of the form
\[
V_{0}\left(  x\right)  =A\left(  1+B\left\vert x\right\vert ^{\mu+2-s}\right)
^{-\frac{1}{q-2}}\text{ for some }A\in%
\mathbb{R}
\text{, }B>0\text{.}%
\]

\end{theorem}

\section{Preliminaries and some important lemmas}

 To carry through our argument, it is necessary to show our new Kelvin type transform can indeed be used to reduce 
 the CKN inequalities with more complicated weights to simpler ones and vice versa. This interchange is nicely done through the following lemmas which are of independent interests and can be found useful in other settings as well.

\begin{lemma}
\label{l2.1}We have that $\left\vert x\cdot\nabla u\left(  x\right)
\right\vert =\left\vert x\right\vert \left\vert \nabla u\left(  x\right)
\right\vert $ for a.e. $x\in%
\mathbb{R}
^{N}$ if and only if $u$ is radially symmetric, that is $u\left(  x\right)
=u\left(  y\right)  $ when $\left\vert x\right\vert =\left\vert y\right\vert
.$

\begin{proof}
If $u$ is radial, then we have
\[
\frac{\partial u}{\partial x_{j}}(x)=u^{\prime}\left(  \left\vert x\right\vert
\right)  \frac{x_{j}}{\left\vert x\right\vert }.
\]
Hence,%
\[
\left\vert \nabla u\left(  x\right)  \right\vert =\left\vert u^{\prime}\left(
\left\vert x\right\vert \right)  \right\vert .
\]
Also,
\[
\left\vert \sum_{j=1}^{N}x_{j}\frac{\partial u}{\partial x_{j}}(x)\right\vert
=\left\vert u^{\prime}\left(  \left\vert x\right\vert \right)  \right\vert
\left\vert \sum_{j=1}^{N}x_{j}\frac{x_{j}}{\left\vert x\right\vert
}\right\vert =\left\vert u^{\prime}\left(  \left\vert x\right\vert \right)
\right\vert \left\vert x\right\vert =\left\vert x\right\vert \left\vert \nabla
u\left(  x\right)  \right\vert .
\]
Now, assume that for all $x:$%
\[
\left\vert x\cdot\nabla u\left(  x\right)  \right\vert =\left\vert
x\right\vert \left\vert \nabla u\left(  x\right)  \right\vert .
\]
It means that $\nabla u\left(  x\right)  $ has the same direction with $x$.
That is we can find a scalar function $g\left(  x\right)  $ such that%
\[
\nabla u\left(  x\right)  =g\left(  x\right)  x.
\]
Now, let $a$ and $b$ be two points on the sphere with radius $r>0$ (that is
$\left\vert a\right\vert =\left\vert b\right\vert =r$). We connect $x$ and $y$
by a piecewise smooth curve $r\left(  t\right)  $ on this sphere, i.e.
$\left\vert r\left(  t\right)  \right\vert =r$, $r\left(  0\right)  =a$ and
$r\left(  1\right)  =b$. Then we have%
\[
\nabla u\left(  r\left(  t\right)  \right)  =g\left(  r\left(  t\right)
\right)  r\left(  t\right)  .
\]
Noting that from%
\[
\left\vert r\left(  t\right)  \right\vert =r\text{ for all }t
\]
we can get that%
\[
\nabla r\left(  t\right)  \cdot r\left(  t\right)  =0.
\]
Thus
\[
\nabla u\left(  r\left(  t\right)  \right)  \cdot\nabla r\left(  t\right)
=g\left(  r\left(  t\right)  \right)  r\left(  t\right)  \cdot\nabla r\left(
t\right)  =0.
\]
So
\[
u\left(  b\right)  -u\left(  a\right)  =u\left(  r\left(  1\right)  \right)
-u\left(  r\left(  0\right)  \right)  =%
{\displaystyle\int\limits_{0}^{1}}
\nabla u\left(  r\left(  t\right)  \right)  \cdot\nabla r\left(  t\right)
dt=0.
\]

This completes the proof of the lemma.
\end{proof}
\end{lemma}

Let $d>1$. We define the vector-valued function $L_{N,d}:%
\mathbb{R}
^{N}\rightarrow%
\mathbb{R}
^{N}$ by
\[
L_{N,d}(x)=|x|^{d-1}x.
\]
The Jacobian matrix of this function $L_{N,d}$ is
\[
\mathbf{J_{L_{N,d}}}=\left(
\begin{array}
[c]{cccc}%
|x|^{d-1}+\left(  d-1\right)  |x|^{d-3}x_{1}^{2} & \left(  d-1\right)
|x|^{d-3}x_{1}x_{2} & \ldots & \left(  d-1\right)  |x|^{d-3}x_{1}x_{N}\\
\left(  d-1\right)  |x|^{d-3}x_{2}x_{1} & |x|^{d-1}+\left(  d-1\right)
|x|^{d-3}x_{2}^{2} & \ldots & \left(  d-1\right)  |x|^{d-3}x_{2}x_{N}\\
\vdots & \vdots & \ddots & \vdots\\
\left(  d-1\right)  |x|^{d-3}x_{N}x_{1} & \left(  d-1\right)  |x|^{d-3}%
x_{N}x_{2} & \ldots & |x|^{d-1}+\left(  d-1\right)  |x|^{d-3}x_{N}^{2}%
\end{array}
\right)  .
\]
We will now show that
\begin{equation}
\det(J_{L_{N,d}})=d|x|^{N\left(  d-1\right)  }. \label{2.2}%
\end{equation}
Indeed, consider the matrix%
\[
\mathbf{A}=\left(
\begin{array}
[c]{cccc}%
\left(  d-1\right)  |x|^{d-3}x_{1}^{2} & \left(  d-1\right)  |x|^{d-3}%
x_{1}x_{2} & \ldots & \left(  d-1\right)  |x|^{d-3}x_{1}x_{N}\\
\left(  d-1\right)  |x|^{d-3}x_{2}x_{1} & \left(  d-1\right)  |x|^{d-3}%
x_{2}^{2} & \ldots & \left(  d-1\right)  |x|^{d-3}x_{2}x_{N}\\
\vdots & \vdots & \ddots & \vdots\\
\left(  d-1\right)  |x|^{d-3}x_{N}x_{1} & \left(  d-1\right)  |x|^{d-3}%
x_{N}x_{2} & \ldots & \left(  d-1\right)  |x|^{d-3}x_{N}^{2}%
\end{array}
\right)  .
\]
It is easy to check that
\[
rank\left(  A\right)  =1\text{ and }tr\left(  A\right)  =\left(  d-1\right)
|x|^{d-1}\text{.}%
\]
Hence, its characteristic polynomial is
\[
\det\left(  \lambda\mathbb{I}_{N}-A\right)  =\lambda^{N}-\left(  d-1\right)
|x|^{d-1}\lambda^{N-1}.
\]
Choose $\lambda=-|x|^{d-1}$, we get%
\[
\det(J_{L_{N,d}})=\left(  -1\right)  ^{N}\det\left(  -|x|^{d-1}\mathbb{I}%
_{N}-A\right)  =d|x|^{N\left(  d-1\right)  }.
\]
We now define mappings $D_{N,d,p}$ with $p>1$ by
\begin{equation}
D_{N,d,p}u(x):=\left(  \frac{1}{d}\right)  ^{\frac{p-1}{p}}u(L_{N,d}%
(x))=\left(  \frac{1}{d}\right)  ^{\frac{p-1}{p}}u(|x|^{d-1}x). \label{2.3}%
\end{equation}
\newline We also define $D_{N,d,p}^{-1}$
\[
D_{N,d,p}^{-1}u=v\text{ if }u=D_{N,d,p}v.
\]

\smallskip

Under the transform $D_{N,d,p}$, we have the following result that will play
an important part in our paper:

\begin{lemma}
\label{l1}(1) For continuous function $f$, we have
\[
\int\limits_{%
\mathbb{R}
^{N}}\frac{f\left(  \left(  \frac{1}{d}\right)  ^{\frac{p-1}{p}}u\left(
x\right)  \right)  }{\left\vert x\right\vert ^{t}}dx=d\int\limits_{%
\mathbb{R}
^{N}}\frac{f\left(  D_{N,d,p}u\left(  x\right)  \right)  }{\left\vert
x\right\vert ^{N+td-Nd}}dx.
\]
In particular, we obtain that $u\in L^{s}\left(  \frac{dx}{\left\vert
x\right\vert ^{t}}\right)  $ if and only if $D_{N,d,p}u\in L^{s}\left(
\frac{dx}{\left\vert x\right\vert ^{N+td-Nd}}\right)  $.

(2) If $\nabla u\in L^{p}\left(  \frac{dx}{\left\vert x\right\vert ^{\mu}%
}\right)  $, then $\nabla D_{N,d,p}u\in L^{p}\left(  \frac{dx}{|x|^{d\left(
p+\mu-N\right)  +N-p}}\right)  $. Moreover,
\[
\int\limits_{%
\mathbb{R}
^{N}}\frac{|\nabla D_{N,d,p}u(x)|^{p}}{|x|^{d\left(  p+\mu-N\right)  +N-p}%
}dx\leq\int\limits_{%
\mathbb{R}
^{N}}\frac{|\nabla u(x)|^{p}}{\left\vert x\right\vert ^{\mu}}dx.
\]
The equality occurs if and only if $u$ is radially symmetric.

\begin{proof}
(1)
\[
\int\limits_{%
\mathbb{R}
^{N}}\frac{f\left(  D_{N,d,p}u\left(  x\right)  \right)  }{\left\vert
x\right\vert ^{N+td-Nd}}dx=\int\limits_{%
\mathbb{R}
^{N}}\frac{f\left(  \left(  \frac{1}{d}\right)  ^{\frac{p-1}{p}}%
u(|x|^{d-1}x)\right)  }{\left\vert x\right\vert ^{N+td-Nd}}dx.
\]
Using change of variables $y_{i}=|x|^{d-1}x_{i}$, $i=1,2,...,N$, we have
\begin{equation}
dy=\det(J_{L_{N,d}})dx=d|x|^{N\left(  d-1\right)  }dx, \label{equ1.7}%
\end{equation}
and
\begin{equation}
dx=\frac{1}{d}\frac{dy}{|y|^{N\frac{d-1}{d}}}. \label{equ1.8}%
\end{equation}
Hence%
\[
\int\limits_{%
\mathbb{R}
^{N}}\frac{f\left(  D_{N,d,p}u\left(  x\right)  \right)  }{\left\vert
x\right\vert ^{N+td-Nd}}dx=\frac{1}{d}\int\limits_{\mathbb{R}^{N}}%
\frac{f\left(  \left(  \frac{1}{d}\right)  ^{\frac{p-1}{p}}u\left(  y\right)
\right)  }{\left\vert y\right\vert ^{N\frac{d-1}{d}}\left\vert y\right\vert
^{\frac{N+td-Nd}{d}}}dy=\frac{1}{d}\int\limits_{%
\mathbb{R}
^{N}}\frac{f\left(  \left(  \frac{1}{d}\right)  ^{\frac{p-1}{p}}u\left(
y\right)  \right)  }{\left\vert y\right\vert ^{t}}dy.
\]

(2) Now we begin to consider the gradient of $D_{N,d,p}u$. After calculations,
we have
\begin{align*}
\left(
\begin{array}
[c]{c}%
\frac{\partial D_{N,d,p}u}{\partial x_{1}}(x)\\
\frac{\partial D_{N,d,p}u}{\partial x_{2}}(x)\\
\vdots\\
\frac{\partial D_{N,d,p}u}{\partial x_{N}}(x)
\end{array}
\right)   &  =\nabla D_{N,d,p}u(x)=\left(  \frac{1}{d}\right)  ^{\frac{p-1}%
{p}}\nabla(u(|x|^{d-1}x))\\
&  =\left(  \frac{1}{d}\right)  ^{\frac{p-1}{p}}J_{L_{N,d}}^{T}\left(
\begin{array}
[c]{c}%
\frac{\partial u}{\partial x_{1}}(|x|^{d-1}x)\\
\frac{\partial u}{\partial x_{2}}(|x|^{d-1}x)\\
\vdots\\
\frac{\partial u}{\partial x_{N}}(|x|^{d-1}x)
\end{array}
\right)  .
\end{align*}
Hence we have
\[
\frac{\partial D_{N,d,p}u}{\partial x_{i}}(x)=\left(  \frac{1}{d}\right)
^{\frac{p-1}{p}}\left(  |x|^{d-1}\frac{\partial u}{\partial x_{i}}%
(|x|^{d-1}x)+A_{i}\right)  ,
\]
for $i=1,2,...N$, where
\[
A_{i}:=\sum_{j=1}^{N}\left(  d-1\right)  |x|^{d-3}x_{i}x_{j}\frac{\partial
u}{\partial x_{j}}(|x|^{d-1}x).
\]
Hence, we obtain
\begin{align}
|\nabla D_{N,d,p}u(x)|^{2}  &  =\sum_{i=1}^{N}\left(  \frac{\partial
D_{N,d,p}u}{\partial x_{i}}(x)\right)  ^{2}\nonumber\\
&  =d^{-2\frac{p-1}{p}}\sum_{i=1}^{N}\left(  |x|^{d-1}\frac{\partial
u}{\partial x_{i}}(|x|^{d-1}x)+A_{i}\right)  ^{2}\nonumber\\
&  =d^{-2\frac{p-1}{p}}\left[  \sum_{i=1}^{N}|x|^{2\left(  d-1\right)
}\left(  \frac{\partial u}{\partial x_{i}}(|x|^{d-1}x)\right)  ^{2}+\sum
_{i=1}^{N}2A_{i}|x|^{d-1}\frac{\partial u}{\partial x_{i}}(|x|^{d-1}%
x)+\sum_{i=1}^{N}A_{i}^{2}\right] \nonumber\\
&  :=d^{-2\frac{p-1}{p}}\left(  I_{1}+I_{2}+I_{3}\right)  .\nonumber
\end{align}
Direct computations show
\[
I_{1}=\sum_{i=1}^{N}|x|^{2\left(  d-1\right)  }\left(  \frac{\partial
u}{\partial x_{i}}(|x|^{d-1}x)\right)  ^{2}=|x|^{2\left(  d-1\right)
}\left\vert \nabla u(|x|^{\frac{t}{N-t}}x)\right\vert ^{2}.
\]
Applying the Cauchy-Schwarz inequality to estimate the second term, we get
\begin{align}
I_{2}  &  =\sum_{i=1}^{N}2A_{i}|x|^{d-1}\frac{\partial u}{\partial x_{i}%
}(|x|^{d-1}x)\nonumber\\
&  =\sum_{i=1}^{N}2|x|^{d-1}\frac{\partial u}{\partial x_{i}}(|x|^{d-1}%
x)\sum_{j=1}^{N}\left(  d-1\right)  |x|^{d-3}x_{i}x_{j}\frac{\partial
u}{\partial x_{j}}(|x|^{d-1}x)\\
&  =2\left(  d-1\right)  |x|^{2d-2}\sum_{i=1}^{N}\sum_{j=1}^{N}\frac
{x_{i}x_{j}}{|x|^{2}}\frac{\partial u}{\partial x_{j}}(|x|^{d-1}%
x)\frac{\partial u}{\partial x_{i}}(|x|^{d-1}x)\nonumber\\
&  =2\left(  d-1\right)  |x|^{2d-2}\left(  \sum_{i=1}^{N}\frac{x_{i}}%
{|x|}\frac{\partial u}{\partial x_{i}}(|x|^{d-1}x)\right)  ^{2}\nonumber\\
&  \leq2\left(  d-1\right)  |x|^{2d-2}\left[  \sum_{i=1}^{N}\left(
\frac{x_{i}}{|x|}\right)  ^{2}\right]  \left[  \sum_{i=1}^{N}\left(
\frac{\partial u}{\partial x_{i}}(|x|^{d-1}x)\right)  ^{2}\right] \nonumber\\
&  =2\left(  d-1\right)  |x|^{2d-2}\left\vert \nabla u(|x|^{d-1}x)\right\vert
^{2}.\nonumber
\end{align}
Similarly for the last term, we have
\begin{align*}
I_{3}=\sum_{i=1}^{N}A_{i}^{2}  &  =\sum_{i=1}^{N}\left(  \sum_{j=1}^{N}\left(
d-1\right)  |x|^{d-3}x_{i}x_{j}\frac{\partial u}{\partial x_{j}}%
(|x|^{d-1}x)\right)  ^{2}\\
&  \leq\left(  d-1\right)  ^{2}|x|^{2d-6}\sum_{i=1}^{N}\left[  \sum_{j=1}%
^{N}\left(  x_{i}x_{j}\right)  ^{2}\right]  \left[  \sum_{j=1}^{N}\left(
\frac{\partial u}{\partial x_{j}}(|x|^{\frac{t}{N-t}}x)\right)  ^{2}\right] \\
&  =\left(  d-1\right)  ^{2}|x|^{2d-6}\sum_{i=1}^{N}\left\vert x\right\vert
^{2}x_{i}^{2}\left\vert \nabla u(|x|^{d-1}x)\right\vert ^{2}\\
&  =\left(  d-1\right)  ^{2}|x|^{2d-2}\left\vert \nabla u(|x|^{d-1}%
x)\right\vert ^{2}.
\end{align*}
Combining them together, we have
\[
|\nabla D_{N,d,p}u(x)|^{2}\leq d^{-2\frac{p-1}{p}}d^{2}|x|^{2d-2}\left\vert
\nabla u(|x|^{d-1}x)\right\vert ^{2}.
\]
This leads to
\[
\left\vert \nabla D_{N,d,p}u(x)\right\vert \leq d^{\frac{1}{p}}|x|^{d-1}%
\left\vert \nabla u(|x|^{d-1}x)\right\vert .
\]
Using the change of variables again, we get%
\begin{align*}
\int\limits_{%
\mathbb{R}
^{N}}\frac{|\nabla u(y)|^{p}}{\left\vert y\right\vert ^{\mu}}dy  &
=\int\limits_{%
\mathbb{R}
^{N}}\frac{|\nabla u(|x|^{d-1}x)|^{p}}{\left\vert |x|^{d-1}x\right\vert ^{\mu
}}d|x|^{N\left(  d-1\right)  }dx\\
&  \geq\frac{1}{d}\int\limits_{%
\mathbb{R}
^{N}}\frac{|\nabla D_{N,d,p}u(x)|^{p}}{|x|^{p\left(  d-1\right)  }\left\vert
|x|^{d-1}x\right\vert ^{\mu}}d|x|^{N\left(  d-1\right)  }dx\\
&  =\int\limits_{%
\mathbb{R}
^{N}}\frac{|\nabla D_{N,d,p}u(x)|^{p}}{|x|^{d\left(  p+\mu-N\right)  +N-p}}dx.
\end{align*}
Finally, by Lemma \ref{l2.1}, it is easy to check that the equalities hold if
and only if $u$ is radial.
\end{proof}
\end{lemma}

\section{Caffarelli-Kohn-Nirenberg inequality when $0<a<1$ under the condition
(\ref{C1})}

Theorem \ref{T4}, Theorem \ref{T5} and Theorem \ref{T6} will be proved via the
following series of lemmas. For the convenience of the reader, we recall that
the conditions on the parameters are
\begin{align}
1  &  <p<p+\mu<N\text{, }\theta\leq\frac{N\mu}{N-p}\leq s<N,\tag{C1}\\
1  &  \leq q<r<\frac{Np}{N-p};\text{ }a=\frac{\left[  \left(  N-\theta\right)
r-\left(  N-s\right)  q\right]  p}{\left[  \left(  N-\theta\right)  p-\left(
N-\mu-p\right)  q\right]  r}.\nonumber
\end{align}
Also,
\[
CKN\left(  N,\mu,\theta,s,p,q,r\right)  =\sup_{u\in D_{\mu,\theta}%
^{p,q}\left(
\mathbb{R}
^{N}\right)  }\frac{\left(  \int\limits_{%
\mathbb{R}
^{N}}\left\vert u\right\vert ^{r}\frac{dx}{\left\vert x\right\vert ^{s}%
}\right)  ^{1/r}}{\left(  \int\limits_{%
\mathbb{R}
^{N}}\left\vert \nabla u\right\vert ^{p}\frac{dx}{\left\vert x\right\vert
^{\mu}}\right)  ^{\frac{a}{p}}\left(  \int\limits_{%
\mathbb{R}
^{N}}\left\vert u\right\vert ^{q}\frac{dx}{\left\vert x\right\vert ^{\theta}%
}\right)  ^{\frac{1-a}{q}}}.
\]
We now set
\[
GN\left(  N,p,q,r,\mu,\theta,s\right)  =\sup_{u\in D_{0,N+\theta d-Nd}%
^{p,q}\left(
\mathbb{R}
^{N}\right)  }\frac{\left(  \int\limits_{%
\mathbb{R}
^{N}}\frac{\left\vert u\right\vert ^{r}}{\left\vert x\right\vert ^{N+sd-Nd}%
}dx\right)  ^{1/r}}{\left(  \int\limits_{%
\mathbb{R}
^{N}}\left\vert \nabla u\right\vert ^{p}dx\right)  ^{\frac{a}{p}}\left(
\int\limits_{%
\mathbb{R}
^{N}}\frac{\left\vert u\right\vert ^{q}}{\left\vert x\right\vert ^{N+\theta
d-Nd}}dx\right)  ^{\frac{1-a}{q}}},
\]
where
\[
d=\frac{N-p}{N-p-\mu}.
\]
It is important to note here that since $\theta\leq\frac{N\mu}{N-p}\leq s<N$,
we have
\[
N+\theta d-Nd\leq0\leq N+sd-Nd<N.
\]

\begin{lemma}
\label{l4.1}The variational problem
\begin{align*}
&  A\left(  N,p,q,r,\mu,\theta,s\right) \\
&  =\inf\left\{  I\left(  u\right)  =\frac{1}{p}\int\limits_{%
\mathbb{R}
^{N}}\left\vert \nabla u\right\vert ^{p}dx+\frac{1}{q}\int\limits_{%
\mathbb{R}
^{N}}\frac{\left\vert u\right\vert ^{q}}{\left\vert x\right\vert ^{N+\theta
d-Nd}}dx:u\in D_{0,N+\theta d-Nd}^{p,q}\left(
\mathbb{R}
^{N}\right)  \text{ and}\int\limits_{%
\mathbb{R}
^{N}}\frac{\left\vert u\right\vert ^{r}}{\left\vert x\right\vert ^{N+sd-Nd}%
}dx=1\right\}
\end{align*}
has a minimizer. Moreover, $A\left(  N,p,q,r,\mu,\theta,s\right)  >0.$

\begin{proof}
By the classical Schwarz rearrangement, we can assume that there exists a
sequence of radial functions $\left(  u_{n}\right)  :$%
\begin{align*}
I\left(  u_{n}\right)   &  \downarrow A\left(  N,p,q,r,\mu,\theta,s\right) \\
\int\limits_{%
\mathbb{R}
^{N}}\frac{\left\vert u_{n}\right\vert ^{r}}{\left\vert x\right\vert
^{N+sd-Nd}}dx  &  =1.
\end{align*}
We can assume WLOG that $u_{n}\rightharpoonup u$ in $u\in D_{0,N+\theta
d-Nd}^{p,q}\left(
\mathbb{R}
^{N}\right)  $. Since it is clear that $I\left(  u\right)  \leq A\left(
N,p,q,r,\mu,\theta,s\right)  $, it is now enough to show
\[
\int\limits_{%
\mathbb{R}
^{N}}\frac{\left\vert u\right\vert ^{r}}{\left\vert x\right\vert ^{N+sd-Nd}%
}dx=1.
\]
But this is easy to observe since we can write for $R>0$ sufficiently large:
\[
\int\limits_{%
\mathbb{R}
^{N}}\frac{\left\vert u_{n}-u\right\vert ^{r}}{\left\vert x\right\vert
^{N+sd-Nd}}dx=%
{\displaystyle\int\limits_{B_{R}}}
+%
{\displaystyle\int\limits_{B_{R}^{c}}}
\frac{\left\vert u_{n}-u\right\vert ^{r}}{\left\vert x\right\vert ^{N+sd-Nd}%
}dx.
\]
Then by the Radial Lemma, we get%
\[%
{\displaystyle\int\limits_{B_{R}^{c}}}
\frac{\left\vert u_{n}-u\right\vert ^{r}}{\left\vert x\right\vert ^{N+sd-Nd}%
}dx\rightarrow0\text{.}%
\]
Also, by the compactness of Sobolev embeddings, we can deduce%
\[%
{\displaystyle\int\limits_{B_{R}}}
\frac{\left\vert u_{n}-u\right\vert ^{r}}{\left\vert x\right\vert ^{N+sd-Nd}%
}dx\rightarrow0\text{.}%
\]
As a consequence, $u\neq0$ and $A\left(  N,p,q,r,\mu,\theta,s\right)  >0.$
Moreover, noting that for $\lambda>0:$
\[
u_{\lambda}\left(  x\right)  =\lambda^{\frac{Nd-sd}{r}}u\left(  \lambda
x\right)  ,
\]
then%
\begin{align*}
\left\Vert \nabla u_{\lambda}\right\Vert _{p}  &  =\lambda^{\frac{Nd-sd}%
{r}+\frac{p-N}{p}}\left\Vert \nabla u\right\Vert _{p};\\
\left\Vert u_{\lambda}\right\Vert _{k}  &  =\lambda^{\frac{Nd-sd}{r}-\frac
{N}{k}}\left\Vert u\right\Vert _{k},
\end{align*}
and
\[
\int\limits_{%
\mathbb{R}
^{N}}\frac{\left\vert u_{\lambda}\right\vert ^{r}}{\left\vert x\right\vert
^{N+sd-Nd}}dx=1.
\]
Also,%
\begin{align*}
I\left(  u_{\lambda}\right)   &  =\frac{1}{p}\int\limits_{%
\mathbb{R}
^{N}}\left\vert \nabla u_{\lambda}\right\vert ^{p}dx+\frac{1}{q}\int\limits_{%
\mathbb{R}
^{N}}\frac{\left\vert u_{\lambda}\right\vert ^{q}}{\left\vert x\right\vert
^{N+\theta d-Nd}}dx\\
&  =\frac{1}{p}\lambda^{\frac{Nd-sd}{r}p+p-N}\left\Vert \nabla u\right\Vert
_{p}^{p}+\frac{1}{q}\lambda^{q\frac{Nd-sd}{r}-N+N+\theta d-Nd}\int\limits_{%
\mathbb{R}
^{N}}\frac{\left\vert u\right\vert ^{q}}{\left\vert x\right\vert ^{N+\theta
d-Nd}}dx\\
&  =\lambda^{m}A+\lambda^{-n}B
\end{align*}
with%
\begin{align*}
m  &  =\frac{Nd-sd}{r}p+p-N\text{; }n=Nd-\theta d-q\frac{Nd-sd}{r}\\
A  &  =\frac{1}{p}\left\Vert \nabla u\right\Vert _{p}^{p};\text{ }B=\frac
{1}{q}\int\limits_{%
\mathbb{R}
^{N}}\frac{\left\vert u\right\vert ^{q}}{\left\vert x\right\vert ^{N+\theta
d-Nd}}dx.
\end{align*}
Hence
\[
A\left(  N,p,q,r,\mu,\theta,s\right)  =\inf_{\lambda>0}I\left(  u_{\lambda
}\right)  =I\left(  u_{\lambda_{0}}\right)
\]
where
\[
\lambda_{0}=\left(  \frac{nB}{mA}\right)  ^{\frac{1}{m+n}}.
\]
It means that
\[
A\left(  N,p,q,r,\mu,\theta,s\right)  =\frac{m+n}{m}\left(  \frac{n}%
{m}\right)  ^{-\frac{n}{m+n}}A^{\frac{n}{m+n}}B^{\frac{m}{m+n}}.
\]

\end{proof}
\end{lemma}

\begin{lemma}
\label{l4.2}$GN\left(  N,p,q,r\right)  $ can be achieved and
\[
GN\left(  N,p,q,r,\mu,\theta,s\right)  =\left[  \frac{\frac{m+n}{m}\left(
\frac{n}{m}\right)  ^{-\frac{n}{m+n}}\left(  \frac{1}{p}\right)  ^{\frac
{n}{m+n}}\left(  \frac{1}{q}\right)  ^{\frac{m}{m+n}}}{A\left(  N,p,q,r,\mu
,\theta,s\right)  }\right]  ^{\frac{\frac{a}{p}}{\frac{n}{m+n}}}.
\]

\begin{proof}
For any $v$ with
\[
\int\limits_{%
\mathbb{R}
^{N}}\frac{\left\vert v\right\vert ^{r}}{\left\vert x\right\vert ^{N+sd-Nd}%
}dx=1,
\]
we use the above process and get
\begin{align*}
A\left(  N,p,q,r,\mu,\theta,s\right)   &  \leq\inf_{\lambda>0}I\left(
v_{\lambda}\right) \\
&  =\frac{m+n}{m}\left(  \frac{n}{m}\right)  ^{-\frac{n}{m+n}}\left(  \frac
{1}{p}\left\Vert \nabla v\right\Vert _{p}^{p}\right)  ^{\frac{n}{m+n}}\left(
\frac{1}{q}\int\limits_{%
\mathbb{R}
^{N}}\frac{\left\vert u\right\vert ^{q}}{\left\vert x\right\vert ^{N+\theta
d-Nd}}dx\right)  ^{\frac{m}{m+n}}.
\end{align*}
Noting that
\[
\frac{\frac{n}{m+n}}{\frac{m}{m+n}}=\frac{\frac{a}{p}}{\frac{1-a}{q}},
\]
we obtain%
\[
\frac{\left(  \int\limits_{%
\mathbb{R}
^{N}}\frac{\left\vert v\right\vert ^{r}}{\left\vert x\right\vert ^{N+sd-Nd}%
}dx\right)  ^{1/r}}{\left(  \int\limits_{%
\mathbb{R}
^{N}}\left\vert \nabla v\right\vert ^{p}dx\right)  ^{\frac{a}{p}}\left(
\int\limits_{%
\mathbb{R}
^{N}}\frac{\left\vert v\right\vert ^{q}}{\left\vert x\right\vert ^{N+\theta
d-Nd}}dx\right)  ^{\frac{1-a}{q}}}\leq\left[  \frac{\frac{m+n}{m}\left(
\frac{n}{m}\right)  ^{-\frac{n}{m+n}}\left(  \frac{1}{p}\right)  ^{\frac
{n}{m+n}}\left(  \frac{1}{q}\right)  ^{\frac{m}{m+n}}}{A\left(  N,p,q,r,\mu
,\theta,s\right)  }\right]  ^{\frac{\frac{a}{p}}{\frac{n}{m+n}}}.
\]
Combining with the previous lemma, we conclude that $GN\left(  N,p,q,r,\mu
,\theta,s\right)  $ can be achieved and%
\[
GN\left(  N,p,q,r,\mu,\theta,s\right)  =\left[  \frac{\frac{m+n}{m}\left(
\frac{n}{m}\right)  ^{-\frac{n}{m+n}}\left(  \frac{1}{p}\right)  ^{\frac
{n}{m+n}}\left(  \frac{1}{q}\right)  ^{\frac{m}{m+n}}}{A\left(  N,p,q,r,\mu
,\theta,s\right)  }\right]  ^{\frac{\frac{a}{p}}{\frac{n}{m+n}}}.
\]

\end{proof}
\end{lemma}

Using Lemma \ref{l4.1} and Lemma \ref{l4.2}, we will now show that $CKN\left(
N,\mu,\theta,s,p,q,r\right)  $ can be achieved:

\begin{lemma}
\label{l4.3}Under (C1), $CKN\left(  N,\mu,\theta,s,p,q,r\right)  $ can be
attained and
\[
CKN\left(  N,\mu,\theta,s,p,q,r\right)  =\left(  \frac{N-p}{N-p-\mu}\right)
^{\frac{1}{r}+\frac{p-1}{p}-\frac{1-a}{q}-\frac{p-1}{p}\left(  1-a\right)
}GN\left(  N,p,q,r,\mu,\theta,s\right)  .
\]

\begin{proof}
We begin by an observation that if $u\geq0$ is a maximizer for $GN\left(
N,p,q,r,\mu,\theta,s\right)  $, then we can assume that $u$ is radial. Indeed,
that fact is just a consequence of the Schwarz rearrangement. See for instance
\cite{LiLo}. Now, let us assume that $U_{0}\geq0$ is a radial maximizer of
$GN\left(  N,p,q,r,\mu,\theta,s\right)  $. We set $V_{0}=D_{N,d,p}^{-1}U_{0}$
with $d=\frac{N-p}{N-p-\mu}$. It means that $U_{0}=D_{N,d,p}V_{0}$. We will
show that $V_{0}$ is a maximizer of $CKN\left(  N,\mu,\theta,s,p,q,r\right)
$. Indeed, for any $v$, we need to show
\[
\frac{\left(  \int\limits_{%
\mathbb{R}
^{N}}\left\vert v\right\vert ^{r}\frac{dx}{\left\vert x\right\vert ^{s}%
}\right)  ^{1/r}}{\left(  \int\limits_{%
\mathbb{R}
^{N}}\left\vert \nabla v\right\vert ^{p}\frac{dx}{\left\vert x\right\vert
^{\mu}}\right)  ^{\frac{a}{p}}\left(  \int\limits_{%
\mathbb{R}
^{N}}\left\vert v\right\vert ^{q}\frac{dx}{\left\vert x\right\vert ^{\theta}%
}\right)  ^{\frac{1-a}{q}}}\leq\frac{\left(  \int\limits_{%
\mathbb{R}
^{N}}\left\vert V_{0}\right\vert ^{r}\frac{dx}{\left\vert x\right\vert ^{s}%
}\right)  ^{1/r}}{\left(  \int\limits_{%
\mathbb{R}
^{N}}\left\vert \nabla V_{0}\right\vert ^{p}\frac{dx}{\left\vert x\right\vert
^{\mu}}\right)  ^{\frac{a}{p}}\left(  \int\limits_{%
\mathbb{R}
^{N}}\left\vert V_{0}\right\vert ^{q}\frac{dx}{\left\vert x\right\vert
^{\theta}}\right)  ^{\frac{1-a}{q}}}.
\]
By Lemma \ref{l1}, we get by noting that when $d=\frac{N-p}{N-p-\mu}$:
$d\left(  p+\mu-N\right)  +N-p=0:$
\[
\int\limits_{%
\mathbb{R}
^{N}}\left\vert v\right\vert ^{r}\frac{dx}{\left\vert x\right\vert ^{s}%
}=d^{1+\frac{p-1}{p}r}\int\limits_{%
\mathbb{R}
^{N}}\frac{\left\vert D_{N,d,p}v\right\vert ^{r}}{\left\vert x\right\vert
^{N+sd-Nd}}dx
\]%
\[
\int\limits_{%
\mathbb{R}
^{N}}\left\vert v\right\vert ^{q}\frac{dx}{\left\vert x\right\vert ^{\theta}%
}=d^{1+\frac{p-1}{p}q}\int\limits_{%
\mathbb{R}
^{N}}\frac{\left\vert D_{N,d,p}v\right\vert ^{q}}{\left\vert x\right\vert
^{N+\theta d-Nd}}dx
\]%
\[
\int\limits_{%
\mathbb{R}
^{N}}\frac{|\nabla v|^{p}}{\left\vert x\right\vert ^{\mu}}dx\geq\int\limits_{%
\mathbb{R}
^{N}}\left\vert \nabla D_{N,d,p}v\right\vert ^{p}dx.
\]
and
\[
\int\limits_{%
\mathbb{R}
^{N}}\left\vert V_{0}\right\vert ^{r}\frac{dx}{\left\vert x\right\vert ^{s}%
}=d^{1+\frac{p-1}{p}r}\int\limits_{%
\mathbb{R}
^{N}}\frac{\left\vert U_{0}\right\vert ^{r}}{\left\vert x\right\vert
^{N+sd-Nd}}dx
\]%
\[
\int\limits_{%
\mathbb{R}
^{N}}\left\vert V_{0}\right\vert ^{q}\frac{dx}{\left\vert x\right\vert
^{\theta}}=d^{1+\frac{p-1}{p}q}\int\limits_{%
\mathbb{R}
^{N}}\frac{\left\vert U_{0}\right\vert ^{q}}{\left\vert x\right\vert
^{N+\theta d-Nd}}dx
\]%
\[
\int\limits_{%
\mathbb{R}
^{N}}\frac{|\nabla V_{0}|^{p}}{\left\vert x\right\vert ^{\mu}}dx=\int\limits_{%
\mathbb{R}
^{N}}\left\vert \nabla U_{0}\right\vert ^{p}dx.
\]
Hence%
\begin{align*}
&  \frac{\left(  \int\limits_{%
\mathbb{R}
^{N}}\left\vert v\right\vert ^{r}\frac{dx}{\left\vert x\right\vert ^{s}%
}\right)  ^{1/r}}{\left(  \int\limits_{%
\mathbb{R}
^{N}}\left\vert \nabla v\right\vert ^{p}\frac{dx}{\left\vert x\right\vert
^{\mu}}\right)  ^{\frac{a}{p}}\left(  \int\limits_{%
\mathbb{R}
^{N}}\left\vert v\right\vert ^{q}\frac{dx}{\left\vert x\right\vert ^{\theta}%
}\right)  ^{\frac{1-a}{q}}}\\
&  \leq\frac{\left(  d^{1+\frac{p-1}{p}r}\int\limits_{%
\mathbb{R}
^{N}}\frac{\left\vert D_{N,d,p}v\right\vert ^{r}}{\left\vert x\right\vert
^{N+sd-Nd}}dx\right)  ^{1/r}}{\left(  \int\limits_{%
\mathbb{R}
^{N}}\left\vert \nabla D_{N,d,p}v\right\vert ^{p}dx\right)  ^{\frac{a}{p}%
}\left(  d^{1+\frac{p-1}{p}q}\int\limits_{%
\mathbb{R}
^{N}}\frac{\left\vert D_{N,d,p}v\right\vert ^{q}}{\left\vert x\right\vert
^{N+\theta d-Nd}}dx\right)  ^{\frac{1-a}{q}}}\\
&  \leq d^{\frac{1}{r}+\frac{p-1}{p}-\frac{1-a}{q}-\frac{p-1}{p}\left(
1-a\right)  }\frac{\left(  \int\limits_{%
\mathbb{R}
^{N}}\frac{\left\vert U_{0}\right\vert ^{r}}{\left\vert x\right\vert
^{N+sd-Nd}}dx\right)  ^{1/r}}{\left(  \int\limits_{%
\mathbb{R}
^{N}}\left\vert \nabla U_{0}\right\vert ^{p}dx\right)  ^{\frac{a}{p}}\left(
\int\limits_{%
\mathbb{R}
^{N}}\frac{\left\vert U_{0}\right\vert ^{q}}{\left\vert x\right\vert
^{N+\theta d-Nd}}dx\right)  ^{\frac{1-a}{q}}}\\
&  =\frac{\left(
{\displaystyle\int}
\left\vert V_{0}\right\vert ^{r}\frac{dx}{\left\vert x\right\vert ^{s}%
}\right)  ^{1/r}}{\left(
{\displaystyle\int}
\left\vert \nabla V_{0}\right\vert ^{p}\frac{dx}{\left\vert x\right\vert
^{\mu}}\right)  ^{\frac{a}{p}}\left(
{\displaystyle\int}
\left\vert V_{0}\right\vert ^{q}\frac{dx}{\left\vert x\right\vert ^{\theta}%
}\right)  ^{\frac{1-a}{q}}}.
\end{align*}
We note that we have the equality in the last row because $U_{0}$ and $V_{0}$
are radial. Hence $CKN\left(  N,\mu,\theta,s,p,q,r\right)  $ is attained.
Moreover, it is easy to see that%
\[
CKN\left(  N,\mu,\theta,s,p,q,r\right)  =d^{\frac{1}{r}+\frac{p-1}{p}%
-\frac{1-a}{q}-\frac{p-1}{p}\left(  1-a\right)  }GN\left(  N,p,q,r,\mu
,\theta,s\right)
\]

\end{proof}
\end{lemma}

\begin{lemma}
\label{l4.4}Assume (C1). If $V_{0}$ is a maximizer of $CKN\left(  N,\mu
,\theta,s,p,q,r\right)  $, then $V_{0}$ is radially symmetric.

\begin{proof}
let $V_{0}$ be a maximizer of $CKN\left(  N,\mu,\theta,s,p,q,r\right)  $. We
set $U_{0}=D_{N,d,p}V_{0}$ where $d=\frac{N-p}{N-p-\mu}$. We will show that
$U_{0}$ is a maximizer of $GN\left(  N,p,q,r,\mu,\theta,s\right)  $. Indeed
for any radial function $u$ (we can just choose radial functions because of
the symmetrization arguments), we define
\[
v=D_{N,d,p}^{-1}u\text{, i.e. }u=D_{N,d,p}v.
\]
By Lemma \ref{l1}, we get%
\[
\int\limits_{%
\mathbb{R}
^{N}}\left\vert v\right\vert ^{r}\frac{dx}{\left\vert x\right\vert ^{s}%
}=d^{1+\frac{p-1}{p}r}\int\limits_{%
\mathbb{R}
^{N}}\frac{\left\vert u\right\vert ^{r}}{\left\vert x\right\vert ^{N+sd-Nd}%
}dx
\]%
\[
\int\limits_{%
\mathbb{R}
^{N}}\left\vert v\right\vert ^{q}\frac{dx}{\left\vert x\right\vert ^{\theta}%
}=d^{1+\frac{p-1}{p}q}\int\limits_{%
\mathbb{R}
^{N}}\frac{\left\vert u\right\vert ^{q}}{\left\vert x\right\vert ^{N+\theta
d-Nd}}dx
\]%
\[
\int\limits_{%
\mathbb{R}
^{N}}\frac{|\nabla v|^{p}}{\left\vert x\right\vert ^{\mu}}dx=\int\limits_{%
\mathbb{R}
^{N}}\left\vert \nabla u\right\vert ^{p}dx
\]
and%
\[
\int\limits_{%
\mathbb{R}
^{N}}\left\vert V_{0}\right\vert ^{r}\frac{dx}{\left\vert x\right\vert ^{s}%
}=d^{1+\frac{p-1}{p}r}\int\limits_{%
\mathbb{R}
^{N}}\frac{\left\vert U_{0}\right\vert ^{r}}{\left\vert x\right\vert
^{N+sd-Nd}}dx
\]%
\[
\int\limits_{%
\mathbb{R}
^{N}}\left\vert V_{0}\right\vert ^{q}\frac{dx}{\left\vert x\right\vert
^{\theta}}=d^{1+\frac{p-1}{p}q}\int\limits_{%
\mathbb{R}
^{N}}\frac{\left\vert U_{0}\right\vert ^{q}}{\left\vert x\right\vert
^{N+\theta d-Nd}}dx
\]%
\[
\int\limits_{%
\mathbb{R}
^{N}}\frac{|\nabla V_{0}|^{p}}{\left\vert x\right\vert ^{\mu}}dx\geq
\int\limits_{%
\mathbb{R}
^{N}}\left\vert \nabla U_{0}\right\vert ^{p}dx.
\]
Hence%
\begin{align*}
&  d^{\frac{1}{r}+\frac{p-1}{p}-\frac{1-a}{q}-\frac{p-1}{p}\left(  1-a\right)
}\frac{\left(  \int\limits_{%
\mathbb{R}
^{N}}\frac{\left\vert U_{0}\right\vert ^{r}}{\left\vert x\right\vert
^{N+sd-Nd}}dx\right)  ^{1/r}}{\left(  \int\limits_{%
\mathbb{R}
^{N}}\left\vert \nabla U_{0}\right\vert ^{p}dx\right)  ^{\frac{a}{p}}\left(
\int\limits_{%
\mathbb{R}
^{N}}\frac{\left\vert U_{0}\right\vert ^{q}}{\left\vert x\right\vert
^{N+\theta d-Nd}}dx\right)  ^{\frac{1-a}{q}}}\\
&  \geq\frac{\left(  \int\limits_{%
\mathbb{R}
^{N}}\left\vert V_{0}\right\vert ^{r}\frac{dx}{\left\vert x\right\vert ^{s}%
}\right)  ^{1/r}}{\left(  \int\limits_{%
\mathbb{R}
^{N}}\frac{|\nabla V_{0}|^{p}}{\left\vert x\right\vert ^{\mu}}dx\right)
^{\frac{a}{p}}\left(  \int\limits_{%
\mathbb{R}
^{N}}\left\vert V_{0}\right\vert ^{q}\frac{dx}{\left\vert x\right\vert
^{\theta}}\right)  ^{\frac{1-a}{q}}}\\
&  \geq\frac{\left(  \int\limits_{%
\mathbb{R}
^{N}}\left\vert v\right\vert ^{r}\frac{dx}{\left\vert x\right\vert ^{s}%
}\right)  ^{1/r}}{\left(  \int\limits_{%
\mathbb{R}
^{N}}\frac{|\nabla v|^{p}}{\left\vert x\right\vert ^{\mu}}dx\right)
^{\frac{a}{p}}\left(  \int\limits_{%
\mathbb{R}
^{N}}\left\vert v\right\vert ^{q}\frac{dx}{\left\vert x\right\vert ^{\theta}%
}\right)  ^{\frac{1-a}{q}}}\\
&  =d^{\frac{1}{r}+\frac{p-1}{p}-\frac{1-a}{q}-\frac{p-1}{p}\left(
1-a\right)  }\frac{\left(  \int\limits_{%
\mathbb{R}
^{N}}\frac{\left\vert u\right\vert ^{r}}{\left\vert x\right\vert ^{N+sd-Nd}%
}dx\right)  ^{1/r}}{\left(  \int\limits_{%
\mathbb{R}
^{N}}\left\vert \nabla u\right\vert ^{p}dx\right)  ^{\frac{a}{p}}\left(
\int\limits_{%
\mathbb{R}
^{N}}\frac{\left\vert u\right\vert ^{q}}{\left\vert x\right\vert ^{N+\theta
d-Nd}}dx\right)  ^{\frac{1-a}{q}}}.
\end{align*}
Moreover, it is easy to see that
\begin{align*}
&  d^{\frac{1}{r}+\frac{p-1}{p}-\frac{1-a}{q}-\frac{p-1}{p}\left(  1-a\right)
}\frac{\left(  \int\limits_{%
\mathbb{R}
^{N}}\frac{\left\vert U_{0}\right\vert ^{r}}{\left\vert x\right\vert
^{N+sd-Nd}}dx\right)  ^{1/r}}{\left(  \int\limits_{%
\mathbb{R}
^{N}}\left\vert \nabla U_{0}\right\vert ^{p}dx\right)  ^{\frac{a}{p}}\left(
\int\limits_{%
\mathbb{R}
^{N}}\frac{\left\vert U_{0}\right\vert ^{q}}{\left\vert x\right\vert
^{N+\theta d-Nd}}dx\right)  ^{\frac{1-a}{q}}}\\
&  =\frac{\left(  \int\limits_{%
\mathbb{R}
^{N}}\left\vert V_{0}\right\vert ^{r}\frac{dx}{\left\vert x\right\vert ^{s}%
}\right)  ^{1/r}}{\left(  \int\limits_{%
\mathbb{R}
^{N}}\frac{|\nabla V_{0}|^{p}}{\left\vert x\right\vert ^{\mu}}dx\right)
^{\frac{a}{p}}\left(  \int\limits_{%
\mathbb{R}
^{N}}\left\vert V_{0}\right\vert ^{q}\frac{dx}{\left\vert x\right\vert
^{\theta}}\right)  ^{\frac{1-a}{q}}}.
\end{align*}
Hence
\[
\int\limits_{%
\mathbb{R}
^{N}}\frac{|\nabla V_{0}|^{p}}{\left\vert x\right\vert ^{\mu}}dx=\int\limits_{%
\mathbb{R}
^{N}}\left\vert \nabla U_{0}\right\vert ^{p}dx.
\]
So $V_{0}$ is radial.
\end{proof}
\end{lemma}

\begin{lemma}
\label{l4.5}Assume (C1) with $s=\theta=\frac{N\mu}{N-p}$. If $r=p\frac
{q-1}{p-1}$, then with $\delta=Np-q\left(  N-p\right)  :$
\begin{align*}
&  CKN\left(  N,s,\mu,p,q,r\right) \\
&  =\left(  \frac{N-p}{N-p-\mu}\right)  ^{\frac{1}{r}+\frac{p-1}{p}-\frac
{1-a}{q}-\frac{p-1}{p}\left(  1-a\right)  }\\
&  \times\left(  \frac{q-p}{p\sqrt{\pi}}\right)  ^{a}\left(  \frac
{pq}{N\left(  q-p\right)  }\right)  ^{\frac{a}{p}}\left(  \frac{\delta}%
{pq}\right)  ^{\frac{1}{r}}\left(  \frac{\Gamma\left(  q\frac{p-1}%
{q-p}\right)  \Gamma\left(  \frac{N}{2}+1\right)  }{\Gamma\left(  \frac
{p-1}{p}\frac{\delta}{q-p}\right)  \Gamma\left(  N\frac{p-1}{p}+1\right)
}\right)  ^{\frac{a}{N}}%
\end{align*}
and all the maximizers have the form
\[
V_{0}\left(  x\right)  =A\left(  1+B\left\vert x\right\vert ^{\frac{1}{d}%
\frac{p}{p-1}}\right)  ^{-\frac{p-1}{q-p}}\text{ for some }A\in%
\mathbb{R}
\text{, }B>0
\]
where $d=\frac{N-p}{N-p-\mu}.$

\begin{proof}
When $r=p\frac{q-1}{p-1}$ and $s=\theta=\frac{N\mu}{N-p}$, we have from
\cite{Agu1, Agu2, DePDo1, DePDo} that
\begin{align*}
&  GN\left(  N,p,q,r\right) \\
&  =\left(  \frac{q-p}{p\sqrt{\pi}}\right)  ^{a}\left(  \frac{pq}{N\left(
q-p\right)  }\right)  ^{\frac{a}{p}}\left(  \frac{\delta}{pq}\right)
^{\frac{1}{r}}\left(  \frac{\Gamma\left(  q\frac{p-1}{q-p}\right)
\Gamma\left(  \frac{N}{2}+1\right)  }{\Gamma\left(  \frac{p-1}{p}\frac{\delta
}{q-p}\right)  \Gamma\left(  N\frac{p-1}{p}+1\right)  }\right)  ^{\frac{a}{N}}%
\end{align*}
and all the maximizers have the form%
\[
U_{0}\left(  x\right)  =A\left(  1+B\left\vert x-\overline{x}\right\vert
^{\frac{p}{p-1}}\right)  ^{-\frac{p-1}{q-p}}\text{ for some }A\in%
\mathbb{R}
\text{, }B>0\text{, }\overline{x}\in%
\mathbb{R}
^{N}\text{.}%
\]
Now, let $V_{0}$ be a maximizer of $CKN\left(  N,s,\mu,p,q,r\right)  $. By
Lemma \ref{l4.4}, $D_{N,d,p}V_{0}$ is a maximizer of $GN\left(
N,p,q,r\right)  $. Hence%
\[
D_{N,d,p}V_{0}\left(  x\right)  =A\left(  1+B\left\vert x-\overline
{x}\right\vert ^{\frac{p}{p-1}}\right)  ^{-\frac{p-1}{q-p}}.
\]
It means that%
\[
V_{0}\left(  x\right)  =A^{\prime}\left(  1+B\left\vert \left\vert
x\right\vert ^{\frac{1}{d}-1}x-\overline{x}\right\vert ^{\frac{p}{p-1}%
}\right)  ^{-\frac{p-1}{q-p}}\text{.}%
\]
Noting that $V_{0}$ is radial, we conclude that $\overline{x}=0$.
\end{proof}
\end{lemma}

\begin{lemma}
\label{l4.6}Assume (C1) with $s=\theta=\frac{N\mu}{N-p}$. If $q=p\frac
{r-1}{p-1}$, then with $\delta=Np-r\left(  N-p\right)  :$
\begin{align*}
&  CKN\left(  N,s,\mu,p,q,r\right) \\
&  =\left(  \frac{N-p}{N-p-\mu}\right)  ^{\frac{1}{r}+\frac{p-1}{p}-\frac
{1-a}{q}-\frac{p-1}{p}\left(  1-a\right)  }\\
&  \times\left(  \frac{p-r}{p\sqrt{\pi}}\right)  ^{a}\left(  \frac
{pr}{N\left(  p-r\right)  }\right)  ^{\frac{a}{p}}\left(  \frac{pr}{\delta
}\right)  ^{\frac{1-a}{q}}\left(  \frac{\Gamma\left(  \frac{p-1}{p}%
\frac{\delta}{p-r}+1\right)  \Gamma\left(  \frac{N}{2}+1\right)  }%
{\Gamma\left(  r\frac{p-1}{p-r}+1\right)  \Gamma\left(  N\frac{p-1}%
{p}+1\right)  }\right)  ^{\frac{a}{N}}.
\end{align*}
If $r>2-\frac{1}{p}$, then all the maximizers of $GN\left(  N,p,q,r\right)  $
have the form%
\[
V_{0}\left(  x\right)  =A\left(  1-B\left\vert x\right\vert ^{\frac{N-p-\mu
}{N-p}\frac{p}{p-1}}\right)  _{+}^{-\frac{p-1}{r-p}}\text{ for some }A\in%
\mathbb{R}
\text{, }B>0\text{.}%
\]

\begin{proof}
When $q=p\frac{r-1}{p-1}$ and $s=\theta=\frac{N\mu}{N-p}$, we have from
\cite{Agu1, Agu2, DePDo1, DePDo} that
\begin{align*}
&  GN\left(  N,p,q,r\right)  \\
&  =\left(  \frac{p-r}{p\sqrt{\pi}}\right)  ^{a}\left(  \frac{pr}{N\left(
p-r\right)  }\right)  ^{\frac{a}{p}}\left(  \frac{pr}{\delta}\right)
^{\frac{1-a}{q}}\left(  \frac{\Gamma\left(  \frac{p-1}{p}\frac{\delta}%
{p-r}+1\right)  \Gamma\left(  \frac{N}{2}+1\right)  }{\Gamma\left(
r\frac{p-1}{p-r}+1\right)  \Gamma\left(  N\frac{p-1}{p}+1\right)  }\right)
^{\frac{a}{N}}.
\end{align*}
Also, when $r>2-\frac{1}{p}$, all the maximizers of $GN\left(  N,p,q,r\right)
$ have the form%
\[
U_{0}\left(  x\right)  =A\left(  1-B\left\vert x-\overline{x}\right\vert
^{\frac{p}{p-1}}\right)  _{+}^{-\frac{p-1}{r-p}}\text{ for some }A\in%
\mathbb{R}
\text{, }B>0\text{, }\overline{x}\in%
\mathbb{R}
^{N}\text{.}%
\]
Now, let $V_{0}$ be a maximizer of $CKN\left(  N,s,\mu,p,q,r\right)  $. By
Lemma \ref{l4.4}, $D_{N,d,p}V_{0}$ is a maximizer of $GN\left(
N,p,q,r\right)  $. Hence%
\[
D_{N,d,p}V_{0}\left(  x\right)  =A\left(  1-B\left\vert x-\overline
{x}\right\vert ^{\frac{p}{p-1}}\right)  _{+}^{-\frac{p-1}{r-p}}.
\]
It means that%
\[
V_{0}\left(  x\right)  =A^{\prime}\left(  1-B\left\vert \left\vert
x\right\vert ^{\frac{1}{d}-1}x-\overline{x}\right\vert ^{\frac{p}{p-1}%
}\right)  _{+}^{-\frac{p-1}{r-p}}\text{.}%
\]
Noting that $V_{0}$ is radially symmetric, we conclude that $\overline{x}=0$.
\end{proof}
\end{lemma}

\section{CKN inequalities in the regime (\ref{C4})}

In this section, we will concern the CKN inequalities in the class
(\ref{C4}):
\begin{align}
p &  =2<2+\mu<N\text{, }2<r=2(q-1)<\frac{2N}{N-2}\tag{C2}\\
\mu+2 &  >s=\theta>\frac{N\mu}{N-2}\text{; }a=\frac{\left(  N-s\right)
\left[  q-2\right]  }{\left[  \left(  N-s\right)  2-\left(  N-\mu-2\right)
q\right]  (q-1)}\nonumber
\end{align}
Recall that%
\[
CKN\left(  N,\mu,s,q\right)  =\sup_{u\in D_{\mu,s}^{2,q}\left(
\mathbb{R}
^{N}\right)  }\frac{\left(  \int\limits_{%
\mathbb{R}
^{N}}\left\vert u\right\vert ^{2(q-1)}\frac{dx}{\left\vert x\right\vert ^{s}%
}\right)  ^{\frac{1}{2(q-1)}}}{\left(  \int\limits_{%
\mathbb{R}
^{N}}\left\vert \nabla u\right\vert ^{2}\frac{dx}{\left\vert x\right\vert
^{\mu}}\right)  ^{\frac{a}{2}}\left(  \int\limits_{%
\mathbb{R}
^{N}}\left\vert u\right\vert ^{q}\frac{dx}{\left\vert x\right\vert ^{s}%
}\right)  ^{\frac{1-a}{q}}}.
\]
We also denote%
\[
CKN_{1}\left(  N,\mu,s,q\right)  =\sup_{u\in D_{0,s}^{2,q}\left(
\mathbb{R}
^{N}\right)  }\frac{\left(  \int\limits_{%
\mathbb{R}
^{N}}\left\vert u\right\vert ^{2(q-1)}\frac{dx}{\left\vert x\right\vert
^{N+sd-Nd}}\right)  ^{\frac{1}{2(q-1)}}}{\left(  \int\limits_{%
\mathbb{R}
^{N}}\left\vert \nabla u\right\vert ^{2}dx\right)  ^{\frac{a}{2}}\left(
\int\limits_{%
\mathbb{R}
^{N}}\left\vert u\right\vert ^{q}\frac{dx}{\left\vert x\right\vert ^{N+sd-Nd}%
}\right)  ^{\frac{1-a}{q}}}%
\]
where $d=\frac{N-2}{N-2-\mu}$. 

\begin{proof}
[Proof of Theorem \ref{T11}]For any $v\in D_{\mu,s}^{2,q}\left(
\mathbb{R}
^{N}\right)  $, we have with $u=D_{N,d,2}v$, that%
\[
\int\limits_{%
\mathbb{R}
^{N}}\left\vert \nabla v\right\vert ^{2}\frac{dx}{\left\vert x\right\vert
^{\mu}}\geq\int\limits_{%
\mathbb{R}
^{N}}|\nabla u|^{2}dx
\]%
\[
\int\limits_{%
\mathbb{R}
^{N}}\left\vert v\right\vert ^{2(q-1)}\frac{dx}{\left\vert x\right\vert ^{s}%
}=d^{q}\int\limits_{%
\mathbb{R}
^{N}}\frac{\left\vert u\right\vert ^{2(q-1)}}{\left\vert x\right\vert
^{N+sd-Nd}}dx
\]%
\[
\int\limits_{%
\mathbb{R}
^{N}}\left\vert v\right\vert ^{q}\frac{dx}{\left\vert x\right\vert ^{s}%
}=d^{1+\frac{1}{2}q}\int\limits_{%
\mathbb{R}
^{N}}\frac{\left\vert u\right\vert ^{q}}{\left\vert x\right\vert ^{N+sd-Nd}%
}dx.
\]
By a result in \cite{DMN}, we have that $U\left(  x\right)  =D_{N,d,2}%
V_{0}\left(  x\right)  =C\left(  1+D\left\vert x\right\vert ^{2-N-sd+Nd}%
\right)  ^{-\frac{1}{q-2}}$ for some $C\in%
\mathbb{R}
$, $D>0$ is the maximizer for $CKN_{1}\left(  N,\mu,s,q\right)  $ for
$0<N+sd-Nd<2$ small enough. Hence we have by Lemma \ref{l1} that
\begin{align*}
\frac{\left(  \int\limits_{%
\mathbb{R}
^{N}}\left\vert v\right\vert ^{2(q-1)}\frac{dx}{\left\vert x\right\vert ^{s}%
}\right)  ^{\frac{1}{2(q-1)}}}{\left(  \int\limits_{%
\mathbb{R}
^{N}}\left\vert \nabla v\right\vert ^{2}\frac{dx}{\left\vert x\right\vert
^{\mu}}\right)  ^{\frac{a}{2}}\left(  \int\limits_{%
\mathbb{R}
^{N}}\left\vert v\right\vert ^{q}\frac{dx}{\left\vert x\right\vert ^{s}%
}\right)  ^{\frac{1-a}{q}}} &  \leq\frac{d^{\frac{q}{2(q-1)}}}{d^{\frac
{1+\frac{1}{2}q}{\frac{1-a}{q}}}}\frac{\left(  \int\limits_{%
\mathbb{R}
^{N}}\left\vert u\right\vert ^{2(q-1)}\frac{dx}{\left\vert x\right\vert
^{N+sd-Nd}}\right)  ^{\frac{1}{2(q-1)}}}{\left(  \int\limits_{%
\mathbb{R}
^{N}}\left\vert \nabla u\right\vert ^{2}dx\right)  ^{\frac{a}{2}}\left(
\int\limits_{%
\mathbb{R}
^{N}}\left\vert u\right\vert ^{q}\frac{dx}{\left\vert x\right\vert ^{N+sd-Nd}%
}\right)  ^{\frac{1-a}{q}}}\\
&  \leq\frac{d^{\frac{q}{2(q-1)}}}{d^{\frac{1+\frac{1}{2}q}{\frac{1-a}{q}}}%
}\frac{\left(  \int\limits_{%
\mathbb{R}
^{N}}\left\vert U\right\vert ^{2(q-1)}\frac{dx}{\left\vert x\right\vert
^{N+sd-Nd}}\right)  ^{\frac{1}{2(q-1)}}}{\left(  \int\limits_{%
\mathbb{R}
^{N}}\left\vert \nabla U\right\vert ^{2}dx\right)  ^{\frac{a}{2}}\left(
\int\limits_{%
\mathbb{R}
^{N}}\left\vert U\right\vert ^{q}\frac{dx}{\left\vert x\right\vert ^{N+sd-Nd}%
}\right)  ^{\frac{1-a}{q}}}\\
&  =\frac{\left(  \int\limits_{%
\mathbb{R}
^{N}}\left\vert V_{0}\right\vert ^{2(q-1)}\frac{dx}{\left\vert x\right\vert
^{s}}\right)  ^{\frac{1}{2(q-1)}}}{\left(  \int\limits_{%
\mathbb{R}
^{N}}\left\vert \nabla V_{0}\right\vert ^{2}\frac{dx}{\left\vert x\right\vert
^{\mu}}\right)  ^{\frac{a}{2}}\left(  \int\limits_{%
\mathbb{R}
^{N}}\left\vert V_{0}\right\vert ^{q}\frac{dx}{\left\vert x\right\vert ^{s}%
}\right)  ^{\frac{1-a}{q}}}.
\end{align*}
In other words, $V_{0}$ is the optimizer for $CKN\left(  N,\mu,s,q\right)  $.
\end{proof}

\section{Caffarelli-Kohn-Nirenberg inequality without the interpolation term:
the case $a=1$}

In this section, we will also consider the CKN inequalities without the
interpolation term for all $1<p<N$ and will concern the following range:
\begin{align}
1 &  <p<\text{ }p+\mu<N\text{, }\frac{\mu}{p}\leq\frac{s}{r}<\frac{\mu}%
{p}+1,\tag{C3}\label{C2}\\
r &  =\frac{\left(  N-s\right)  p}{N-\mu-p};\text{ }a=1.\nonumber
\end{align}
Noting that the condition $\frac{\mu}{p}\leq\frac{s}{r}\leq\frac{\mu}{p}+1$
comes from the constraints of the CKN inequality. In this case, we denote
\[
D_{\mu}^{1,p}\left(
\mathbb{R}
^{N};\frac{dx}{\left\vert x\right\vert ^{s}}\right)  =\left\{  u\in
L^{r}\left(  \frac{dx}{\left\vert x\right\vert ^{s}}\right)  :\int\limits_{%
\mathbb{R}
^{N}}\left\vert \nabla u\right\vert ^{p}\frac{dx}{\left\vert x\right\vert
^{\mu}}<\infty\right\}  ,
\]
and%
\[
CKN\left(  N,p,\mu,s\right)  =\sup_{u\in D_{\mu}^{1,p}\left(
\mathbb{R}
^{N};\frac{dx}{\left\vert x\right\vert ^{s}}\right)  }\frac{\left(
\int\limits_{%
\mathbb{R}
^{N}}\left\vert u\right\vert ^{r}\frac{dx}{\left\vert x\right\vert ^{s}%
}\right)  ^{1/r}}{\left(  \int\limits_{%
\mathbb{R}
^{N}}\left\vert \nabla u\right\vert ^{p}\frac{dx}{\left\vert x\right\vert
^{\mu}}\right)  ^{1/p}}.
\]

Then we will prove in this section the following result:

\begin{theorem}
\label{T1}Assume (\ref{C2}). Then $CKN\left(  N,p,\mu,s\right)  $ is achieved
with the extremals being of the following form:
\[
V_{c,\lambda}\left(  x\right)  =c\left(  \lambda+\left\vert x\right\vert
^{\frac{p+\mu-s}{p-1}}\right)  ^{-\frac{N-p-\mu}{p+\mu-s}}%
\]
for some $c\neq0$, $\lambda>0$.
\end{theorem}

Theorem \ref{T1} was studied in \cite{Mu} by solving the corresponding ODE. In
this section, we will provide another proof using the transform $D_{N,d,p}$.

We note that $\frac{\mu}{p}\leq\frac{s}{r}<\frac{\mu}{p}+1$ means%
\[
\frac{N\mu}{N-p}\leq s<p+\mu.
\]

We also denote
\[
HS\left(  N,p,\mu,s\right)  =\sup_{D_{0}^{1,p}\left(
\mathbb{R}
^{N};\frac{dx}{\left\vert x\right\vert ^{\frac{s\left(  N-p\right)  -N\mu
}{N-p-\mu}}}\right)  }\frac{\left(  \int\limits_{%
\mathbb{R}
^{N}}\left\vert u\right\vert ^{p^{\ast}\left(  \frac{s\left(  N-p\right)
-N\mu}{N-p-\mu}\right)  }\frac{dx}{\left\vert x\right\vert ^{\frac{s\left(
N-p\right)  -N\mu}{N-p-\mu}}}\right)  ^{1/p^{\ast}\left(  \frac{s\left(
N-p\right)  -N\mu}{N-p-\mu}\right)  }}{\left(  \int\limits_{%
\mathbb{R}
^{N}}\left\vert \nabla u\right\vert ^{p}dx\right)  ^{1/p}}.
\]
Noting that
\[
p^{\ast}\left(  \frac{s\left(  N-p\right)  -N\mu}{N-p-\mu}\right)
=\frac{N-\frac{s\left(  N-p\right)  -N\mu}{N-p-\mu}}{N-p}p=\frac{\left(
N-s\right)  p}{N-\mu-p}%
\]
and
\[
0\leq\frac{s\left(  N-p\right)  -N\mu}{N-p-\mu}<p.
\]

\begin{lemma}
\label{l3.1}$CKN\left(  N,p,\mu,s\right)  $ can be attained

\begin{proof}
We will use the fact that $HS\left(  N,p,\mu,s\right)  $ is attained by some
radial functions $U_{0}$. Set $V_{0}=D_{N,d,p}^{-1}U_{0}$ with $d=\frac
{N-p}{N-p-\mu}$. It means that $U_{0}=D_{N,d,p}V_{0}$. We will show that
$V_{0}$ is a maximizer of $CKN\left(  N,p,\mu,s\right)  .$ Indeed, for any
$v\in D_{\mu}^{1,p}\left(
\mathbb{R}
^{N}\right)  $, we set $u=D_{N,d,p}v$. Then by Lemma \ref{l1}, we get%
\[
\int\limits_{%
\mathbb{R}
^{N}}\left\vert \nabla v\right\vert ^{p}\frac{dx}{\left\vert x\right\vert
^{\mu}}\geq\int\limits_{%
\mathbb{R}
^{N}}\frac{|\nabla D_{N,d,p}v|^{p}}{|x|^{d\left(  p+\mu-N\right)  +N-p}%
}dx=\int\limits_{%
\mathbb{R}
^{N}}|\nabla u|^{p}dx
\]
\begin{align*}
\int\limits_{%
\mathbb{R}
^{N}}\left\vert v\right\vert ^{\frac{\left(  N-s\right)  p}{N-\mu-p}}\frac
{dx}{\left\vert x\right\vert ^{s}}  &  =dd^{\frac{p-1}{p}\frac{\left(
N-s\right)  p}{N-\mu-p}}\int\limits_{%
\mathbb{R}
^{N}}\left\vert u\right\vert ^{\frac{\left(  N-s\right)  p}{N-\mu-p}}\frac
{dx}{\left\vert x\right\vert ^{N+sd-Nd}}\\
&  =d^{1+\frac{\left(  N-s\right)  \left(  p-1\right)  }{N-\mu-p}}%
\int\limits_{%
\mathbb{R}
^{N}}\left\vert u\right\vert ^{\frac{\left(  N-s\right)  p}{N-\mu-p}}\frac
{dx}{\left\vert x\right\vert ^{\frac{s\left(  N-p\right)  -N\mu}{N-p-\mu}}}.
\end{align*}%
\[
\int\limits_{%
\mathbb{R}
^{N}}\left\vert \nabla V_{0}\right\vert ^{p}\frac{dx}{\left\vert x\right\vert
^{\mu}}=\int\limits_{%
\mathbb{R}
^{N}}|\nabla U_{0}|^{p}dx
\]
and%
\[
\int\limits_{%
\mathbb{R}
^{N}}\left\vert V_{0}\right\vert ^{\frac{\left(  N-s\right)  p}{N-\mu-p}}%
\frac{dx}{\left\vert x\right\vert ^{s}}=d^{1+\frac{\left(  N-s\right)  \left(
p-1\right)  }{N-\mu-p}}\int\limits_{%
\mathbb{R}
^{N}}\left\vert U_{0}\right\vert ^{\frac{\left(  N-s\right)  p}{N-\mu-p}}%
\frac{dx}{\left\vert x\right\vert ^{\frac{s\left(  N-p\right)  -N\mu}{N-p-\mu
}}}.
\]
Hence
\begin{align*}
\frac{\left(  \int\limits_{%
\mathbb{R}
^{N}}\left\vert v\right\vert ^{\frac{\left(  N-s\right)  p}{N-\mu-p}}\frac
{dx}{\left\vert x\right\vert ^{s}}\right)  ^{\frac{N-\mu-p}{\left(
N-s\right)  p}}}{\left(  \int\limits_{%
\mathbb{R}
^{N}}\left\vert \nabla v\right\vert ^{p}\frac{dx}{\left\vert x\right\vert
^{\mu}}\right)  ^{1/p}}  &  \leq d^{\left(  1+\frac{\left(  N-s\right)
\left(  p-1\right)  }{N-\mu-p}\right)  \frac{N-\mu-p}{\left(  N-s\right)  p}%
}\frac{\left(  \int\limits_{%
\mathbb{R}
^{N}}\left\vert u\right\vert ^{\frac{\left(  N-s\right)  p}{N-\mu-p}}\frac
{dx}{\left\vert x\right\vert ^{\frac{s\left(  N-p\right)  -N\mu}{N-p-\mu}}%
}\right)  ^{\frac{N-\mu-p}{\left(  N-s\right)  p}}}{\left(  \int\limits_{%
\mathbb{R}
^{N}}|\nabla u|^{p}dx\right)  ^{1/p}}\\
&  \leq d^{\left(  1+\frac{\left(  N-s\right)  \left(  p-1\right)  }{N-\mu
-p}\right)  \frac{N-\mu-p}{\left(  N-s\right)  p}}\frac{\left(  \int\limits_{%
\mathbb{R}
^{N}}\left\vert U_{0}\right\vert ^{\frac{\left(  N-s\right)  p}{N-\mu-p}}%
\frac{dx}{\left\vert x\right\vert ^{\frac{s\left(  N-p\right)  -N\mu}{N-p-\mu
}}}\right)  ^{\frac{N-\mu-p}{\left(  N-s\right)  p}}}{\left(  \int\limits_{%
\mathbb{R}
^{N}}|\nabla U_{0}|^{p}dx\right)  ^{1/p}}\\
&  =\frac{\left(  \int\limits_{%
\mathbb{R}
^{N}}\left\vert V_{0}\right\vert ^{\frac{\left(  N-s\right)  p}{N-\mu-p}}%
\frac{dx}{\left\vert x\right\vert ^{s}}\right)  ^{\frac{N-\mu-p}{\left(
N-s\right)  p}}}{\left(  \int\limits_{%
\mathbb{R}
^{N}}\left\vert \nabla V_{0}\right\vert ^{p}\frac{dx}{\left\vert x\right\vert
^{\mu}}\right)  ^{1/p}}.
\end{align*}
In other words, $V_{0}$ is the maximizer for $CKN\left(  N,p,\mu,s\right)  .$
Moreover, we also deduce that%
\[
CKN\left(  N,p,\mu,s\right)  =d^{\left(  1+\frac{\left(  N-s\right)  \left(
p-1\right)  }{N-\mu-p}\right)  \frac{N-\mu-p}{\left(  N-s\right)  p}}HS\left(
N,p,\mu,s\right)  .
\]

\end{proof}
\end{lemma}

\begin{lemma}
\label{l3.2}All the optimizers for $CKN\left(  N,p,\mu,s\right)  $ are
radially symmetric.

\begin{proof}
Assume that $V_{0}$ is a maximizer for $CKN\left(  N,p,\mu,s\right)  $ and
$U_{0}=D_{N,d,p}V_{0}$ where $d=\frac{N-p}{N-p-\mu}$. Again by Lemma \ref{l1},
we get%
\[
\int\limits_{%
\mathbb{R}
^{N}}\left\vert \nabla V_{0}\right\vert ^{p}\frac{dx}{\left\vert x\right\vert
^{\mu}}\geq\int\limits_{%
\mathbb{R}
^{N}}|\nabla U_{0}|^{p}dx
\]
and%
\[
\int\limits_{%
\mathbb{R}
^{N}}\left\vert V_{0}\right\vert ^{\frac{\left(  N-s\right)  p}{N-\mu-p}}%
\frac{dx}{\left\vert x\right\vert ^{s}}=d^{1+\frac{\left(  N-s\right)  \left(
p-1\right)  }{N-\mu-p}}\int\limits_{%
\mathbb{R}
^{N}}\left\vert U_{0}\right\vert ^{\frac{\left(  N-s\right)  p}{N-\mu-p}}%
\frac{dx}{\left\vert x\right\vert ^{\frac{s\left(  N-p\right)  -N\mu}{N-p-\mu
}}}.
\]
We will now prove that $U_{0}$ is a maximizer for $HS\left(  N,p,\mu,s\right)
$. Indeed, for any radial function $u$ (we can assume $u$ is radial by the
Schwarz rearrangement argument), and set $v=D_{N,d,p}^{-1}u$, that is
$u=D_{N,d,p}v$, we obtain by Lemma \ref{l1}:%
\[
\int\limits_{%
\mathbb{R}
^{N}}\left\vert \nabla v\right\vert ^{p}\frac{dx}{\left\vert x\right\vert
^{\mu}}=\int\limits_{%
\mathbb{R}
^{N}}|\nabla u|^{p}dx
\]
and%
\[
\int\limits_{%
\mathbb{R}
^{N}}\left\vert v\right\vert ^{\frac{\left(  N-s\right)  p}{N-\mu-p}}\frac
{dx}{\left\vert x\right\vert ^{s}}=d^{1+\frac{\left(  N-s\right)  \left(
p-1\right)  }{N-\mu-p}}\int\limits_{%
\mathbb{R}
^{N}}\left\vert u\right\vert ^{\frac{\left(  N-s\right)  p}{N-\mu-p}}\frac
{dx}{\left\vert x\right\vert ^{\frac{s\left(  N-p\right)  -N\mu}{N-p-\mu}}}.
\]
Hence,
\begin{align*}
\frac{\left(  \int\limits_{%
\mathbb{R}
^{N}}\left\vert U_{0}\right\vert ^{\frac{\left(  N-s\right)  p}{N-\mu-p}}%
\frac{dx}{\left\vert x\right\vert ^{\frac{s\left(  N-p\right)  -N\mu}{N-p-\mu
}}}\right)  ^{\frac{N-\mu-p}{\left(  N-s\right)  p}}}{\left(  \int\limits_{%
\mathbb{R}
^{N}}|\nabla U_{0}|^{p}dx\right)  ^{1/p}}  &  \geq\left(  \frac{1}{d}\right)
^{\left(  1+\frac{\left(  N-s\right)  \left(  p-1\right)  }{N-\mu-p}\right)
\frac{N-\mu-p}{\left(  N-s\right)  p}}\frac{\left(  \int\limits_{%
\mathbb{R}
^{N}}\left\vert V_{0}\right\vert ^{\frac{\left(  N-s\right)  p}{N-\mu-p}}%
\frac{dx}{\left\vert x\right\vert ^{s}}\right)  ^{\frac{N-\mu-p}{\left(
N-s\right)  p}}}{\left(  \int\limits_{%
\mathbb{R}
^{N}}\left\vert \nabla V_{0}\right\vert ^{p}\frac{dx}{\left\vert x\right\vert
^{\mu}}\right)  ^{1/p}}\\
&  \geq\left(  \frac{1}{d}\right)  ^{\left(  1+\frac{\left(  N-s\right)
\left(  p-1\right)  }{N-\mu-p}\right)  \frac{N-\mu-p}{\left(  N-s\right)  p}%
}\frac{\left(  \int\limits_{%
\mathbb{R}
^{N}}\left\vert v\right\vert ^{\frac{\left(  N-s\right)  p}{N-\mu-p}}\frac
{dx}{\left\vert x\right\vert ^{s}}\right)  ^{\frac{N-\mu-p}{\left(
N-s\right)  p}}}{\left(  \int\limits_{%
\mathbb{R}
^{N}}\left\vert \nabla v\right\vert ^{p}\frac{dx}{\left\vert x\right\vert
^{\mu}}\right)  ^{1/p}}\\
&  =\frac{\left(  \int\limits_{%
\mathbb{R}
^{N}}\left\vert u\right\vert ^{\frac{\left(  N-s\right)  p}{N-\mu-p}}\frac
{dx}{\left\vert x\right\vert ^{\frac{s\left(  N-p\right)  -N\mu}{N-p-\mu}}%
}\right)  ^{\frac{N-\mu-p}{\left(  N-s\right)  p}}}{\left(  \int\limits_{%
\mathbb{R}
^{N}}|\nabla u|^{p}dx\right)  ^{1/p}}.
\end{align*}
Hence $U_{0}$ is a maximizer for $HS\left(  N,p,\mu,s\right)  $. Moreover, it
is easy to see that the equality must happen in the first line, that is%
\[
\frac{\left(  \int\limits_{%
\mathbb{R}
^{N}}\left\vert U_{0}\right\vert ^{\frac{\left(  N-s\right)  p}{N-\mu-p}}%
\frac{dx}{\left\vert x\right\vert ^{\frac{s\left(  N-p\right)  -N\mu}{N-p-\mu
}}}\right)  ^{\frac{N-\mu-p}{\left(  N-s\right)  p}}}{\left(  \int\limits_{%
\mathbb{R}
^{N}}|\nabla U_{0}|^{p}dx\right)  ^{1/p}}=\left(  \frac{1}{d}\right)
^{\left(  1+\frac{\left(  N-s\right)  \left(  p-1\right)  }{N-\mu-p}\right)
\frac{N-\mu-p}{\left(  N-s\right)  p}}\frac{\left(  \int\limits_{%
\mathbb{R}
^{N}}\left\vert V_{0}\right\vert ^{\frac{\left(  N-s\right)  p}{N-\mu-p}}%
\frac{dx}{\left\vert x\right\vert ^{s}}\right)  ^{\frac{N-\mu-p}{\left(
N-s\right)  p}}}{\left(  \int\limits_{%
\mathbb{R}
^{N}}\left\vert \nabla V_{0}\right\vert ^{p}\frac{dx}{\left\vert x\right\vert
^{\mu}}\right)  ^{1/p}}.
\]
It means that
\[
\int\limits_{%
\mathbb{R}
^{N}}\left\vert \nabla V_{0}\right\vert ^{p}\frac{dx}{\left\vert x\right\vert
^{\mu}}=\int\limits_{%
\mathbb{R}
^{N}}|\nabla U_{0}|^{p}dx
\]
and thus, $V_{0}$ is radial.
\end{proof}
\end{lemma}

\begin{proof}
[Proof of Theorem \ref{T1}]From Lemma \ref{l3.1} and Lemma \ref{l3.2}, we see
that $CKN\left(  N,p,\mu,s\right)  $ is attained,%
\[
CKN\left(  N,p,\mu,s\right)  =d^{\left(  1+\frac{\left(  N-s\right)  \left(
p-1\right)  }{N-\mu-p}\right)  \frac{N-\mu-p}{\left(  N-s\right)  p}}HS\left(
N,p,\mu,s\right)  ,
\]
and all maximizers for $CKN\left(  N,p,\mu,s\right)  $ are radially symmetric.
Futhermore, we can conclude that $V_{0}$ is a maximizer for $CKN\left(
N,p,\mu,s\right)  $ only if $U_{0}=D_{N,d,p}V_{0}$ is a maximizer for
$HS\left(  N,p,\mu,s\right)  $ where $d=\frac{N-p}{N-p-\mu}$. It is known (see
\cite{GM} for instance) that $HS\left(  N,p,\mu,s\right)  $ is attained with
the maximizers being the functions
\begin{align*}
U_{c,\lambda}\left(  x\right)   &  =c\left(  \lambda+\left\vert x\right\vert
^{\frac{p-\frac{s\left(  N-p\right)  -N\mu}{N-p-\mu}}{p-1}}\right)
^{-\frac{N-p}{p-\frac{s\left(  N-p\right)  -N\mu}{N-p-\mu}}}\\
&  =c\left(  \lambda+\left\vert x\right\vert ^{\frac{\left(  N-p\right)
\left(  p+\mu-s\right)  }{\left(  p-1\right)  \left(  N-p-\mu\right)  }%
}\right)  ^{-\frac{N-p-\mu}{p+\mu-s}}\text{ }%
\end{align*}
for some $c\neq0$, $\lambda>0.$ Hence $CKN\left(  N,p,\mu,s\right)  $ could be
achieved with the optimizers being the functions%
\begin{align*}
V_{c,\lambda}\left(  x\right)   &  =D_{N,d,p}^{-1}U_{c,\lambda}\left(
x\right) \\
&  =c\left(  \lambda+\left\vert x\right\vert ^{\frac{p+\mu-s}{p-1}}\right)
^{-\frac{N-p-\mu}{p+\mu-s}}%
\end{align*}
for some $c\neq0$, $\lambda>0.$
\end{proof}

\begin{remark}
If we have $\frac{s}{r}=\frac{\mu}{p}+1$ in the condition (C2), then $s=p+\mu$
and $\frac{s\left(  N-p\right)  -N\mu}{N-p-\mu}=p$. So in this case, after
applying the transform $D_{N,d,p}$ where $d=\frac{N-p}{N-p-\mu}$, the CKN
inequality corresponds to the Hardy inequality. Hence, the best constant in
this case is
\[
CKN\left(  N,p,\mu,s\right)  =\frac{p}{N-p-\mu},
\]
and it is never achieved.
\end{remark}

\section{Caffarelli-Kohn Nirenberg inequalities with arbitrary norm}

In this section, we will investigate the CKN under arbitrary norms in $%
\mathbb{R}
^{N}$ in the spirit of Cordero-Erausquin, Nazaret and Villani \cite{CeNV}.
More precisely, let $E=\left(
\mathbb{R}
^{N},\left\Vert \cdot\right\Vert \right)  $, where $\left\Vert \cdot
\right\Vert $ is an arbitrary norm on $%
\mathbb{R}
^{N}$. Then its dual space $E^{\ast}=\left(
\mathbb{R}
^{N},\left\Vert \cdot\right\Vert _{\ast}\right)  $ where for $X\in E^{\ast}:$%
\[
\left\Vert X\right\Vert _{\ast}=\sup_{Y\in E:\left\Vert Y\right\Vert \leq
1}X\cdot Y.
\]
For simplicity, we will assume that $\left\vert \left\{  \left\Vert
x\right\Vert _{\ast}\leq1\right\}  \right\vert =\omega_{N}$ and denote
$\kappa_{N}=\left\vert \left\{  \left\Vert x\right\Vert \leq1\right\}
\right\vert .$ We will asume that for any $X\in%
\mathbb{R}
^{N}$, there exists a unique $X^{\ast}\in%
\mathbb{R}
^{N}$ such that $\left\Vert X^{\ast}\right\Vert _{\ast}=1$ and
\[
X\cdot X^{\ast}=\left\Vert X\right\Vert =\sup_{Y\in%
\mathbb{R}
^{N}:\left\Vert Y\right\Vert _{\ast}\leq1}X\cdot Y.
\]
It is clear that $\left\Vert \cdot\right\Vert $ is Lipschitz with the
Lipschitz constant $1$, and thus, differentiable a.e. From the property that%
\[
\left\Vert \lambda x\right\Vert =\lambda\left\Vert x\right\Vert \text{ for all
}\lambda>0\text{,}%
\]
we can see that the gradient of $\left\Vert \cdot\right\Vert $ at $x\in%
\mathbb{R}
^{N}$ is the unique vector $\nabla\left(  \left\Vert \cdot\right\Vert \right)
\left(  x\right)  =x^{\ast}.$ Recall that
\[
\left\Vert x^{\ast}\right\Vert _{\ast}=1\text{, }x\cdot x^{\ast}=\left\Vert
x\right\Vert =\sup_{\left\Vert y\right\Vert _{\ast}\leq1}x\cdot y.
\]
Actually, at first, we will consider a more general situation. More precisely,
we suppose that $C$ is $q-$homogeneous, that is there exists $q>1$ such that%
\begin{equation}
C\left(  \lambda x\right)  =\lambda^{q}C\left(  x\right)  \text{ }%
\forall\lambda\geq0,\text{ }\forall x\in%
\mathbb{R}
^{N}. \label{2.10}%
\end{equation}
Then $C^{\ast}$, the Legendre transform of $C$, defined by
\[
C^{\ast}\left(  x\right)  =\sup_{y}\left\{  \left\langle x,y\right\rangle
-C\left(  y\right)  \right\}  ,
\]
is even, strictly convex function and is $p-$homogeneous with $p=\frac{q}%
{q-1}$.

We have that $\left\langle X,Y\right\rangle \leq C^{\ast}\left(  X\right)
+C(Y)$ for all $X,Y$. Hence $\left\langle X,Y\right\rangle \leq\lambda
^{p}C^{\ast}\left(  X\right)  +\lambda^{-q}C(Y)$ for all $\lambda>0$, $X,Y.$
Minimizing the right hand side with respect to $\lambda$ gives the
Cauchy-Schwarz inequality
\[
X\cdot Y\leq\left[  qC\left(  Y\right)  \right]  ^{\frac{1}{q}}\left[
pC^{\ast}\left(  X\right)  \right]  ^{\frac{1}{p}}.
\]
By Young's inequality, we have%
\[
X\cdot Y\leq\left[  qC\left(  Y\right)  \right]  ^{\frac{1}{q}}\left[
pC^{\ast}\left(  X\right)  \right]  ^{\frac{1}{p}}\leq C^{\ast}\left(
x\right)  +C(y).
\]
Hence, we also have that%
\[
\left[  pC^{\ast}\left(  X\right)  \right]  ^{\frac{1}{p}}=\sup_{Y}%
\frac{X\cdot Y}{\left[  qC\left(  Y\right)  \right]  ^{\frac{1}{q}}}.
\]
In other words,%
\[
C^{\ast}\left(  X\right)  =\sup_{Y}\frac{\left\vert X\cdot Y\right\vert ^{p}%
}{p\left[  qC\left(  Y\right)  \right]  ^{\frac{p}{q}}}%
\]
We will assume that for all $x\in%
\mathbb{R}
^{N}$, there exists a unique vector $x^{\ast}$ such that
\[
x\cdot x^{\ast}=qC\left(  x\right)  \text{ and }C^{\ast}\left(  x^{\ast
}\right)  =\left(  q-1\right)  C\left(  x\right)  =\frac{q}{p}C\left(
x\right)  .
\]
In other words, for all $x\in%
\mathbb{R}
^{N},$ there exists a unique vector $x^{\ast}$ such that the equality in the
Cauchy-Schwarz inequality happens.

Noting that from (\ref{2.10}), we get that $C\left(  \cdot\right)  $ is
differentiable a.e. We will assume that the gradient of $C\left(
\cdot\right)  $ at $x\in%
\mathbb{R}
^{N}$ is the unique vector $x^{\ast}$. The example that we have in mind are
$C\left(  x\right)  =\frac{1}{q}\left\vert x\right\vert ^{q}$ and $C^{\ast
}\left(  x\right)  =\frac{1}{p}\left\vert x\right\vert ^{p}$ with $\left\vert
\cdot\right\vert $ is the regular Euclidean norm on $%
\mathbb{R}
^{N}$. Another example is the pair $C\left(  x\right)  =\frac{1}{q}\left\Vert
x\right\Vert ^{q}$ and $C^{\ast}\left(  x\right)  =\frac{1}{p}\left\Vert
x\right\Vert _{\ast}^{p}.$

\subsection{A change of variables}

Similarly as Lemma \ref{l2.1}, we have

\begin{lemma}
\label{l5.0}We have
\[
\left\vert x\cdot\nabla u(x)\right\vert =\left[  qC\left(  x\right)  \right]
^{\frac{1}{q}}\left[  pC^{\ast}\left(  \nabla u\right)  \right]  ^{\frac{1}%
{p}}\text{ for a.e}.x\in%
\mathbb{R}
^{N}\text{,}%
\]
if and only if $u$ is $C-$radial, i.e. $u\left(  x\right)  =u\left(  y\right)
$ when $C\left(  x\right)  =C\left(  y\right)  $.

\begin{proof}
If $u$ is $C-$radial, then recalling that $\nabla\left(  C\left(
\cdot\right)  \right)  \left(  x\right)  =x^{\ast},$ we have
\[
\frac{\partial u}{\partial x_{j}}(x)=u^{\prime}\left(  C\left(  x\right)
\right)  x_{j}^{\ast}.
\]
Hence,%
\[
C^{\ast}\left(  \nabla u\right)  =C^{\ast}\left(  u^{\prime}\left(  C\left(
x\right)  \right)  x^{\ast}\right)  =\left\vert u^{\prime}\left(  C\left(
x\right)  \right)  \right\vert ^{p}C^{\ast}\left(  x^{\ast}\right)
=\left\vert u^{\prime}\left(  C\left(  x\right)  \right)  \right\vert
^{p}\frac{q}{p}C\left(  x\right)
\]
and%
\[
\left[  qC\left(  x\right)  \right]  ^{\frac{1}{q}}\left[  pC^{\ast}\left(
\nabla u\right)  \right]  ^{\frac{1}{p}}=\left[  qC\left(  x\right)  \right]
^{\frac{1}{q}}\left[  \left\vert u^{\prime}\left(  C\left(  x\right)  \right)
\right\vert ^{p}qC\left(  x\right)  \right]  ^{\frac{1}{p}}=\left\vert
u^{\prime}\left(  C\left(  x\right)  \right)  \right\vert qC\left(  x\right)
.
\]
Also,
\[
\left\vert x\cdot\nabla u(x)\right\vert =\left\vert \sum_{j=1}^{N}x_{j}%
\frac{\partial u}{\partial x_{j}}(x)\right\vert =\left\vert u^{\prime}\left(
\left\Vert x\right\Vert \right)  \right\vert \left\vert \sum_{j=1}^{N}%
x_{j}x_{j}^{\ast}\right\vert =\left\vert u^{\prime}\left(  \left\Vert
x\right\Vert \right)  \right\vert qC\left(  x\right)  .
\]
Now, if for all $x\in%
\mathbb{R}
^{N}:$
\[
\left\vert x\cdot\nabla u(x)\right\vert =\left[  qC\left(  x\right)  \right]
^{\frac{1}{q}}\left[  pC^{\ast}\left(  \nabla u\right)  \right]  ^{\frac{1}%
{p}},
\]
then $\nabla u(x)$ has the same direction with $x^{\ast}$. That is we can find
a function $f\left(  x\right)  $ such that $\nabla u(x)=f\left(  x\right)
x^{\ast}$. Now let $a$ and $b$ be two points on the $C-$sphere with radius
$r>0$. That is $C\left(  a\right)  =C\left(  b\right)  =r$. We connect $a$ and
$b$ by a piecewise smooth curve $r\left(  t\right)  $ on the sphere, i.e.
$C\left(  r\left(  t\right)  \right)  =r$ and $C\left(  r\left(  0\right)
\right)  =a$, $C\left(  r\left(  1\right)  \right)  =b$. Then we have%
\[
\nabla u(r\left(  t\right)  )=f\left(  r\left(  t\right)  \right)  \left(
r\left(  t\right)  \right)  ^{\ast}.
\]
Using that fact that $C\left(  r\left(  t\right)  \right)  =r$ for all $t$, we
get%
\[
\left(  r\left(  t\right)  \right)  ^{\ast}\cdot\nabla r\left(  t\right)
=0\text{.}%
\]
Hence%
\[%
{\displaystyle\int\limits_{0}^{1}}
\nabla u(r\left(  t\right)  )\cdot\nabla r\left(  t\right)  dt=%
{\displaystyle\int\limits_{0}^{1}}
f\left(  r\left(  t\right)  \right)  \left(  r\left(  t\right)  \right)
^{\ast}\cdot\nabla r\left(  t\right)  dt=0.
\]
In other words,
\[
u\left(  b\right)  -u\left(  a\right)  =u\left(  C\left(  r\left(  1\right)
\right)  \right)  -u\left(  C\left(  r\left(  0\right)  \right)  \right)
=0\text{.}%
\]

\end{proof}
\end{lemma}

Let $d>0$. We define the vector-valued function $L_{N,d}:%
\mathbb{R}
^{N}\rightarrow%
\mathbb{R}
^{N}$ by
\[
L_{N,d}(x)=C\left(  x\right)  ^{d}x.
\]
The Jacobian matrix of this function $L_{N,d}$ is
\[
\mathbf{J_{L_{N,d}}}=C\left(  x\right)  ^{d}\mathbb{I}_{N}+A
\]
where%
\[
\mathbf{A}=\left(
\begin{array}
[c]{cccc}%
dC\left(  x\right)  ^{d-1}x_{1}x_{1}^{\ast} & dC\left(  x\right)  ^{d-1}%
x_{1}x_{2}^{\ast} & \ldots & dC\left(  x\right)  ^{d-1}x_{1}x_{N}^{\ast}\\
dC\left(  x\right)  ^{d-1}x_{2}x_{1}^{\ast} & dC\left(  x\right)  ^{d-1}%
x_{2}x_{2}^{\ast} & \ldots & dC\left(  x\right)  ^{d-1}x_{2}x_{N}^{\ast}\\
\vdots & \vdots & \ddots & \vdots\\
dC\left(  x\right)  ^{d-1}x_{N}x_{1}^{\ast} & dC\left(  x\right)  ^{d-1}%
x_{N}x_{2} & \ldots & dC\left(  x\right)  ^{d-1}x_{N}x_{N}^{\ast}%
\end{array}
\right)  .
\]
Then we get%
\[
\det(J_{L_{N,d}})=\left(  -1\right)  ^{N}\det\left(  -C\left(  x\right)
^{d}\mathbb{I}_{N}-A\right)  =\left(  1+dq\right)  C\left(  x\right)  ^{Nd}.
\]
We now also define mappings $D_{N,d,p}$ with $p>1$ by
\[
D_{N,d,p}u(x):=\left(  \frac{1}{1+dq}\right)  ^{\frac{p-1}{p}}u(L_{N,d}%
(x))=\left(  \frac{1}{1+dq}\right)  ^{\frac{p-1}{p}}u(C\left(  x\right)
^{d}x).
\]
\newline We also define $D_{N,d,p}^{-1}$
\[
D_{N,d,p}^{-1}u=v\text{ if }u=D_{N,d,p}v.
\]
\smallskip

Under the transform $D_{N,d,p}$, we also have the following result:

\begin{lemma}
\label{l5.1}(1) For continuous function $f$, we have
\[
\int\limits_{%
\mathbb{R}
^{N}}\frac{f\left(  \left(  \frac{1}{1+dq}\right)  ^{\frac{p-1}{p}}u\left(
x\right)  \right)  }{C\left(  x\right)  ^{t}}dx=\left(  1+dq\right)
\int\limits_{%
\mathbb{R}
^{N}}\frac{f\left(  D_{N,d,p}u\left(  x\right)  \right)  }{C\left(  x\right)
^{t\left(  dq+1\right)  -Nd}}dx.
\]
In particular, we obtain that $u\in L^{s}\left(  \frac{dx}{C\left(  x\right)
^{t}}\right)  $ if and only if $D_{N,d,p}u\in L^{s}\left(  \frac{dx}{C\left(
x\right)  ^{t\left(  dq+1\right)  -Nd}}\right)  $.

(2) For smooth functions $u:$
\[
\int\limits_{%
\mathbb{R}
^{N}}\frac{C^{\ast}\left(  \nabla D_{N,d,p}u(x)\right)  }{C\left(  x\right)
^{\left(  qd+1\right)  \mu+pd-Nd}}dx\leq\int\limits_{%
\mathbb{R}
^{N}}\frac{C^{\ast}\left(  \nabla u(y)\right)  }{C\left(  y\right)  ^{\mu}%
}dy.
\]
The equality occurs if and only if $u$ is $C-$radially symmetric.

\begin{proof}
(1)
\[
\int\limits_{%
\mathbb{R}
^{N}}\frac{f\left(  D_{N,d,p}u\left(  x\right)  \right)  }{C\left(  x\right)
^{t\left(  dq+1\right)  -Nd}}dx=\frac{1}{1+dq}\int\limits_{%
\mathbb{R}
^{N}}\frac{f\left(  \left(  \frac{1}{1+dq}\right)  ^{\frac{p-1}{p}}u\left(
y\right)  \right)  }{C\left(  y\right)  ^{t}}dy.
\]
Using change of variables $y_{i}=C\left(  x\right)  ^{d}x_{i}$, $i=1,2,...,N$,
we have
\[
dy=\det(J_{L_{N,d}})dx=\left(  1+dq\right)  C\left(  x\right)  ^{Nd}dx,
\]
and
\[
dx=\frac{1}{\left(  1+dq\right)  }\frac{dy}{C\left(  y\right)  ^{\frac
{Nd}{dq+1}}}.
\]
Hence%
\begin{align*}
\int\limits_{%
\mathbb{R}
^{N}}\frac{f\left(  D_{N,d,p}u\left(  x\right)  \right)  }{C\left(  x\right)
^{t\left(  dq+1\right)  -Nd}}dx  &  =\int\limits_{%
\mathbb{R}
^{N}}\frac{f\left(  \left(  \frac{1}{1+dq}\right)  ^{\frac{p-1}{p}}u(C\left(
x\right)  ^{d}x)\right)  }{C\left(  x\right)  ^{t\left(  dq+1\right)  -Nd}%
}dx\\
&  =\frac{1}{1+dq}\int\limits_{%
\mathbb{R}
^{N}}\frac{f\left(  \left(  \frac{1}{1+dq}\right)  ^{\frac{p-1}{p}}u\left(
y\right)  \right)  }{C\left(  y\right)  ^{\frac{t\left(  dq+1\right)
-Nd}{dq+1}}}\frac{dy}{C\left(  y\right)  ^{\frac{Nd}{dq+1}}}=\frac{1}%
{1+dq}\int\limits_{%
\mathbb{R}
^{N}}\frac{f\left(  \left(  \frac{1}{1+dq}\right)  ^{\frac{p-1}{p}}u\left(
y\right)  \right)  }{C\left(  y\right)  ^{t}}dy.
\end{align*}

(2) Now we begin to consider the gradient of $D_{N,d,p}u$. After calculations,
we have
\begin{align*}
\left(
\begin{array}
[c]{c}%
\frac{\partial D_{N,d,p}u}{\partial x_{1}}(x)\\
\frac{\partial D_{N,d,p}u}{\partial x_{2}}(x)\\
\vdots\\
\frac{\partial D_{N,d,p}u}{\partial x_{N}}(x)
\end{array}
\right)   &  =\nabla D_{N,d,p}u(x)=\left(  \frac{1}{1+dq}\right)  ^{\frac
{p-1}{p}}\nabla(u(C\left(  x\right)  ^{d}x))\\
&  =\left(  \frac{1}{1+dq}\right)  ^{\frac{p-1}{p}}J_{L_{N,d}}^{T}\left(
\begin{array}
[c]{c}%
\frac{\partial u}{\partial x_{1}}(C\left(  x\right)  ^{d}x)\\
\frac{\partial u}{\partial x_{2}}(C\left(  x\right)  ^{d}x)\\
\vdots\\
\frac{\partial u}{\partial x_{N}}(C\left(  x\right)  ^{d}x)
\end{array}
\right)  .
\end{align*}
Hence we have
\[
\frac{\partial u(C\left(  x\right)  ^{d}x)}{\partial x_{i}}=\left(  C\left(
x\right)  ^{d}\frac{\partial u}{\partial x_{i}}(C\left(  x\right)
^{d}x)+A_{i}\right)  ,
\]
for $i=1,2,...N$, where
\[
A_{i}:=\sum_{j=1}^{N}dC\left(  x\right)  ^{d-1}x_{i}^{\ast}x_{j}\frac{\partial
u}{\partial x_{j}}(C\left(  x\right)  ^{d}x).
\]%
\[
C^{\ast}\left(  X\right)  =\sup\frac{\left\vert X\cdot Y\right\vert ^{p}%
}{p\left[  qC\left(  Y\right)  \right]  ^{\frac{p}{q}}}%
\]
Hence, we obtain%
\begin{align*}
C^{\ast}\left(  \nabla D_{N,d,p}u(x)\right)   &  =C^{\ast}\left(  \left(
\frac{1}{1+dq}\right)  ^{\frac{p-1}{p}}\nabla(u(C\left(  x\right)
^{d}x))\right)  =\left(  \frac{1}{1+dq}\right)  ^{p-1}C^{\ast}\left(
\nabla(u(C\left(  x\right)  ^{d}x))\right) \\
&  =\left(  \frac{1}{1+dq}\right)  ^{p-1}\sup_{y}\left\{  \frac{\left(
\nabla(u(C\left(  x\right)  ^{d}x))\cdot y\right)  ^{p}}{p\left[  qC\left(
y\right)  \right]  ^{\frac{p}{q}}}\right\} \\
&  =\left(  \frac{1}{1+dq}\right)  ^{p-1}\sup_{y}\left\{  \frac{\left[
\sum_{i=1}^{N}\left[  C\left(  x\right)  ^{d}\frac{\partial u}{\partial x_{i}%
}(C\left(  x\right)  ^{d}x)y_{i}+A_{i}y_{i}\right]  \right]  ^{p}}{p\left[
qC\left(  y\right)  \right]  ^{\frac{p}{q}}}\right\}  .
\end{align*}
The first term is easy to compute:
\begin{align*}
I_{1}  &  =\sum_{i=1}^{N}C\left(  x\right)  ^{d}\frac{\partial u}{\partial
x_{i}}(C\left(  x\right)  ^{d}x)y_{i}\\
&  =C\left(  x\right)  ^{d}\nabla u\left(  C\left(  x\right)  ^{d}x\right)
\cdot y\\
&  \leq C\left(  x\right)  ^{d}\left[  qC\left(  y\right)  \right]  ^{\frac
{1}{q}}\left[  pC^{\ast}\left(  \nabla u\left(  C\left(  x\right)
^{d}x\right)  \right)  \right]  ^{\frac{1}{p}}%
\end{align*}
Applying the Cauchy-Schwarz inequality
\[
X\cdot Y\leq\left[  qC\left(  Y\right)  \right]  ^{\frac{1}{q}}\left[
pC^{\ast}\left(  X\right)  \right]  ^{\frac{1}{p}},
\]
we can estimate the second term:
\begin{align*}
I_{2}  &  =\sum_{i=1}^{N}A_{i}y_{i}\\
&  =\sum_{i=1}^{N}\sum_{j=1}^{N}dC\left(  x\right)  ^{d-1}x_{i}^{\ast}%
x_{j}\frac{\partial u}{\partial x_{j}}(C\left(  x\right)  ^{d}x)y_{i}\\
&  =dC\left(  x\right)  ^{d-1}\sum_{i=1}^{N}x_{i}^{\ast}y_{i}\sum_{j=1}%
^{N}x_{j}\frac{\partial u}{\partial x_{j}}(C\left(  x\right)  ^{d}x)\\
&  \leq dC\left(  x\right)  ^{d-1}\left\vert x^{\ast}\cdot y\right\vert
\left\vert x\cdot\nabla u(C\left(  x\right)  ^{d}x)\right\vert \\
&  \leq dC\left(  x\right)  ^{d-1}\left[  qC\left(  y\right)  \right]
^{\frac{1}{q}}\left[  pC^{\ast}\left(  x^{\ast}\right)  \right]  ^{\frac{1}%
{p}}\left[  qC\left(  x\right)  \right]  ^{\frac{1}{q}}\left[  pC^{\ast
}\left(  \nabla u(C\left(  x\right)  ^{d}x)\right)  \right]  ^{\frac{1}{p}}\\
&  \leq dC\left(  x\right)  ^{d-1}\left[  qC\left(  y\right)  \right]
^{\frac{1}{q}}\left[  qC\left(  x\right)  \right]  ^{\frac{1}{p}}\left[
qC\left(  x\right)  \right]  ^{\frac{1}{q}}\left[  pC^{\ast}\left(  \nabla
u(C\left(  x\right)  ^{d}x)\right)  \right]  ^{\frac{1}{p}}\\
&  \leq qdC\left(  x\right)  ^{d}\left[  qC\left(  y\right)  \right]
^{\frac{1}{q}}\left[  pC^{\ast}\left(  \nabla u(C\left(  x\right)
^{d}x)\right)  \right]  ^{\frac{1}{p}}%
\end{align*}
Therefore
\begin{align*}
&  \sup_{y}\left\{  \frac{\left[  \sum_{i=1}^{N}\left[  C\left(  x\right)
^{d}\frac{\partial u}{\partial x_{i}}(C\left(  x\right)  ^{d}x)y_{i}%
+A_{i}y_{i}\right]  \right]  ^{p}}{p\left[  qC\left(  y\right)  \right]
^{\frac{p}{q}}}\right\} \\
&  \leq\sup_{y}\left\{  \frac{\left[  \left(  1+qd\right)  \right]
^{p}C\left(  x\right)  ^{pd}\left[  qC\left(  y\right)  \right]  ^{\frac{p}%
{q}}pC^{\ast}\left(  \nabla u(C\left(  x\right)  ^{d}x)\right)  }{p\left[
qC\left(  y\right)  \right]  ^{\frac{p}{q}}}\right\} \\
&  =\left[  \left(  1+qd\right)  \right]  ^{p}C\left(  x\right)  ^{pd}C^{\ast
}\left(  \nabla u(C\left(  x\right)  ^{d}x)\right)  .
\end{align*}
In conclusion, we get%
\[
C^{\ast}\left(  \nabla D_{N,d,p}u(x)\right)  \leq\left(  1+qd\right)  C\left(
x\right)  ^{pd}C^{\ast}\left(  \nabla u(C\left(  x\right)  ^{d}x)\right)  .
\]
Using the change of variables again, we get%
\begin{align*}
\int\limits_{%
\mathbb{R}
^{N}}\frac{C^{\ast}\left(  \nabla u(y)\right)  }{C\left(  y\right)  ^{\mu}}dy
&  =\int\limits_{%
\mathbb{R}
^{N}}\frac{C^{\ast}\left(  \nabla u(C\left(  x\right)  ^{d}x)\right)
}{C\left(  C\left(  x\right)  ^{d}x\right)  ^{\mu}}\left(  1+dq\right)
C\left(  x\right)  ^{Nd}dx\\
&  \geq\int\limits_{%
\mathbb{R}
^{N}}\frac{C^{\ast}\left(  \nabla D_{N,d,p}u(x)\right)  }{C\left(  x\right)
^{\left(  qd+1\right)  \mu}C\left(  x\right)  ^{pd}}C\left(  x\right)
^{Nd}dx\\
&  =\int\limits_{%
\mathbb{R}
^{N}}\frac{C^{\ast}\left(  \nabla D_{N,d,p}u(x)\right)  }{C\left(  x\right)
^{\left(  qd+1\right)  \mu+pd-Nd}}dx
\end{align*}
Finally, it is easy to check that the equalities hold if and only if the
equality in the Cauchy-Schwarz inequality occur. It means that $u$ is
$C-$radially symmetric.
\end{proof}
\end{lemma}

We note here that we will mainly apply the above change of variables with
$C\left(  x\right)  =\frac{1}{q}\left\Vert x\right\Vert ^{q}$ and $C^{\ast
}\left(  x\right)  =\frac{1}{p}\left\Vert x\right\Vert _{\ast}^{p}.$ In this
case, for the easy references, we will use the transform
\[
T_{N,d,p}u(x):=\left(  \frac{1}{d}\right)  ^{\frac{p-1}{p}}u(\left\Vert
x\right\Vert ^{d-1}x).
\]
\newline We also define $T_{N,d,p}^{-1}$
\[
T_{N,d,p}^{-1}u=v\text{ if }u=T_{N,d,p}v.
\]
The following lemma is a restatement of Lemma \ref{l5.1}.

\begin{lemma}
\label{l1.1}(1) For continuous function $f$, we have
\[
\int\limits_{%
\mathbb{R}
^{N}}\frac{f\left(  \left(  \frac{1}{d}\right)  ^{\frac{p-1}{p}}u\left(
x\right)  \right)  }{\left\Vert x\right\Vert ^{t}}dx=d\int\limits_{%
\mathbb{R}
^{N}}\frac{f\left(  T_{N,d,p}u\left(  x\right)  \right)  }{\left\Vert
x\right\Vert ^{N+td-Nd}}dx.
\]
In particular, we obtain that $u\in L^{s}\left(  \frac{dx}{\left\Vert
x\right\Vert ^{t}}\right)  $ if and only if $T_{N,d,p}u\in L^{s}\left(
\frac{dx}{\left\Vert x\right\Vert ^{N+td-Nd}}\right)  $.

(2) If $\nabla u\in L^{p}\left(  \frac{dx}{\left\Vert x\right\Vert ^{\mu}%
}\right)  $, then $\nabla T_{N,d,p}\in L^{p}\left(  \frac{dx}{\left\Vert
x\right\Vert ^{d\left(  p+\mu-N\right)  +N-p}}\right)  $. Moreover,
\[
\int\limits_{%
\mathbb{R}
^{N}}\frac{\left\Vert \nabla T_{N,d,p}u(x)\right\Vert _{\ast}^{p}}{\left\Vert
x\right\Vert ^{d\left(  p+\mu-N\right)  +N-p}}dx\leq\int\limits_{%
\mathbb{R}
^{N}}\frac{\left\Vert \nabla u(x)\right\Vert _{\ast}^{p}}{\left\Vert
x\right\Vert ^{\mu}}dx.
\]
The equality occurs if and only if $u$ is $\left\Vert \cdot\right\Vert -$radial.
\end{lemma}

\subsection{Maximizers for the CKN inequalities with arbitrary norms}

Consider the following class:%

\begin{align}
1 &  <p<\text{ }p+\mu<N\text{, }\theta\leq\frac{N\mu}{N-p}\leq s<N,\tag{C4}%
\label{C3}\\
1 &  \leq q<r<\frac{Np}{N-p};\text{ }a=\frac{\left[  \left(  N-\theta\right)
r-\left(  N-s\right)  q\right]  p}{\left[  \left(  N-\theta\right)  p-\left(
N-\mu-p\right)  q\right]  r}.\nonumber
\end{align}

Denote $D_{\mu,\theta}^{p,q}\left(
\mathbb{R}
^{N}\right)  $ the completion of the space of smooth compactly supported
functions with the norm $\left(  \int\limits_{%
\mathbb{R}
^{N}}\left\Vert \nabla u\right\Vert _{\ast}^{p}\frac{dx}{\left\Vert
x\right\Vert ^{\mu}}\right)  ^{1/p}+\left(  \int\limits_{%
\mathbb{R}
^{N}}\left\vert u\right\vert ^{q}\frac{dx}{\left\Vert x\right\Vert ^{\theta}%
}\right)  ^{1/q}$, and set
\begin{align*}
CKN\left(  N,\mu,\theta,s,p,q,r\right)   &  =\sup_{u\in D_{\mu,\theta}%
^{p,q}\left(
\mathbb{R}
^{N}\right)  }\frac{\left(  \int\limits_{%
\mathbb{R}
^{N}}\left\vert u\right\vert ^{r}\frac{dx}{\left\Vert x\right\Vert ^{s}%
}\right)  ^{1/r}}{\left(  \int\limits_{%
\mathbb{R}
^{N}}\left\Vert \nabla u\right\Vert _{\ast}^{p}\frac{dx}{\left\Vert
x\right\Vert ^{\mu}}\right)  ^{\frac{a}{p}}\left(  \int\limits_{%
\mathbb{R}
^{N}}\left\vert u\right\vert ^{q}\frac{dx}{\left\Vert x\right\Vert ^{\theta}%
}\right)  ^{\frac{1-a}{q}}};\\
GN\left(  N,p,q,r,\mu,\theta,s\right)   &  =\sup_{u\in D_{0,N+\theta
d-Nd}^{p,q}\left(
\mathbb{R}
^{N}\right)  }\frac{\left(  \int\limits_{%
\mathbb{R}
^{N}}\frac{\left\vert u\right\vert ^{r}}{\left\Vert x\right\Vert ^{N+sd-Nd}%
}dx\right)  ^{1/r}}{\left(  \int\limits_{%
\mathbb{R}
^{N}}\left\Vert \nabla u\right\Vert _{\ast}^{p}dx\right)  ^{\frac{a}{p}%
}\left(  \int\limits_{%
\mathbb{R}
^{N}}\frac{\left\vert u\right\vert ^{q}}{\left\Vert x\right\Vert ^{N+\theta
d-Nd}}dx\right)  ^{\frac{1-a}{q}}}%
\end{align*}
Then, similarly as in Section 3, we can prove that

\begin{theorem}
\label{T7}Assume (C3). Then $CKN\left(  N,\mu,\theta,s,p,q,r\right)  $ can be
achieved. Moreover, all the extremal functions of $CKN\left(  N,\mu
,\theta,s,p,q,r\right)  $ are $\left\Vert \cdot\right\Vert -$radially symmetric.
\end{theorem}

Furthermore, we can provide the maximizers for $CKN\left(  N,\mu
,\theta,s,p,q,r\right)  $ in the following two classes:

\begin{theorem}
\label{T8}Assume (C2) with $\theta=s=\frac{N\mu}{N-p}$. If $r=p\frac{q-1}%
{p-1}$, then $CKN\left(  N,\mu,\theta,s,p,q,r\right)  $ is achieved by
maximizers of the form%
\[
V_{0}\left(  x\right)  =A\left(  1+B\left\Vert x\right\Vert ^{\frac{N-p-\mu
}{N-p}\frac{p}{p-1}}\right)  ^{-\frac{p-1}{q-p}}\text{ for some }A\in%
\mathbb{R}
\text{, }B>0\text{.}%
\]

\end{theorem}

\begin{theorem}
\label{T9}Assume (C2) with $\theta=s=\frac{N\mu}{N-p}$. If $q=p\frac{r-1}%
{p-1}$, then if $r>2-\frac{1}{p}$, $CKN\left(  N,\mu,\theta,s,p,q,r\right)  $
is achieved by maximizers of the form
\[
V_{0}\left(  x\right)  =A\left(  1-B\left\Vert x\right\Vert ^{\frac{N-p-\mu
}{N-p}\frac{p}{p-1}}\right)  _{+}^{-\frac{p-1}{r-p}}\text{ for some }A\in%
\mathbb{R}
\text{, }B>0\text{.}%
\]

\begin{proof}
The proof of Theorem \ref{T7} is similar to that of Theorem \ref{T4} and will
be omitted.

When $r=p\frac{q-1}{p-1}$ and $s=\theta=\frac{N\mu}{N-p}$, we have from
\cite{CeNV} that $GN\left(  N,p,q,r\right)  $ is achieved by maximizers of the
form%
\[
U_{0}\left(  x\right)  =A\left(  1+B\left\Vert x-\overline{x}\right\Vert
^{\frac{p}{p-1}}\right)  ^{-\frac{p-1}{q-p}}\text{ for some }A\in%
\mathbb{R}
\text{, }B>0\text{, }\overline{x}\in%
\mathbb{R}
^{N}\text{.}%
\]
Now, let $V_{0}$ be a maximizer of $CKN\left(  N,s,\mu,p,q,r\right)  $. Then
$T_{N,d,p}V_{0}$ is a maximizer of $GN\left(  N,p,q,r\right)  $ with
$d=\frac{N-p}{N-p-\mu}$. Hence $CKN\left(  N,s,\mu,p,q,r\right)  $ can be
attained by
\[
V_{0}\left(  x\right)  =T_{N,d,p}^{-1}A\left(  1+B\left\Vert x-\overline
{x}\right\Vert ^{\frac{p}{p-1}}\right)  ^{-\frac{p-1}{q-p}}.
\]
It means that%
\[
V_{0}\left(  x\right)  =A^{\prime}\left(  1+B\left\Vert \left\vert
x\right\vert ^{\frac{1}{d}-1}x-\overline{x}\right\Vert ^{\frac{p}{p-1}%
}\right)  ^{-\frac{p-1}{q-p}}\text{.}%
\]
Noting that $V_{0}$ is $\left\Vert \cdot\right\Vert -$radial, we conclude that
$\overline{x}=0$. That is
\[
V_{0}\left(  x\right)  =A^{\prime}\left(  1+B\left\Vert x\right\Vert
^{\frac{N-p-\mu}{N-p}\frac{p}{p-1}}\right)  ^{-\frac{p-1}{q-p}}\text{.}%
\]
Simiarly, when $\theta=s=\frac{N\mu}{N-p}$. If $q=p\frac{r-1}{p-1}$ and
$r>2-\frac{1}{p}$, $CKN\left(  N,\mu,\theta,s,p,q,r\right)  $ is achieved by
maximizers of the form
\[
V_{0}\left(  x\right)  =A\left(  1-B\left\Vert x\right\Vert ^{\frac{N-p-\mu
}{N-p}\frac{p}{p-1}}\right)  _{+}^{-\frac{p-1}{r-p}}\text{ for some }A\in%
\mathbb{R}
\text{, }B>0\text{.}%
\]

\end{proof}
\end{theorem}

\section{Further comments}

Let $d>1$. Then under the transform $T_{N,d,p}$, the CKN inequality with the
triple $\left(  s,\mu, \theta\right)  $ would be converted to the one with the
triple $\left(  N+sd-Nd,d\left(  p+\mu-N\right)  +N-p,N+\theta d-Nd\right)  .$
We should note that%
\begin{align*}
a  &  =\frac{\left[  \left(  N-\theta\right)  r-\left(  N-s\right)  q\right]
p}{\left[  \left(  N-\theta\right)  p-\left(  N-\mu-p\right)  q\right]  r}\\
&  =\frac{\left[  \left(  N-\left(  N+\theta d-Nd\right)  \right)  r-\left(
N-\left(  N+sd-Nd\right)  \right)  q\right]  p}{\left[  \left(  N-\left(
N+\theta d-Nd\right)  \right)  p-\left(  N-\left(  d\left(  p+\mu-N\right)
+N-p\right)  -p\right)  q\right]  r}.
\end{align*}
This fact may be used to simplify the study of symmetry/symmetry breaking
phenomena. For instance, we could prove that

\begin{theorem}
\label{T10}Assume $d=\frac{N-p}{N-p-\mu}>1$ and $0<a=\frac{\left[  \left(
N-\theta\right)  r-\left(  N-s\right)  q\right]  p}{\left[  \left(
N-\theta\right)  p-\left(  N-\mu-p\right)  q\right]  r}\leq1$. If
\[
CKN_{1}=\sup_{u\in D_{0,N+\theta d-Nd}^{p,q}\left(
\mathbb{R}
^{N}\right)  }\frac{\left(  \int\limits_{%
\mathbb{R}
^{N}}\left\vert u\right\vert ^{r}\frac{dx}{\left\Vert x\right\Vert ^{N+sd-Nd}%
}\right)  ^{1/r}}{\left(  \int\limits_{%
\mathbb{R}
^{N}}\left\Vert \nabla u\right\Vert _{\ast}^{p}dx\right)  ^{\frac{a}{p}%
}\left(  \int\limits_{%
\mathbb{R}
^{N}}\left\vert u\right\vert ^{q}\frac{dx}{\left\Vert x\right\Vert ^{N+\theta
d-Nd}}\right)  ^{\frac{1-a}{q}}}%
\]
has a $\left\Vert \cdot\right\Vert -$radially symmetric maximizer, then
\[
CKN_{2}=\sup_{u\in D_{\mu,\theta}^{p,q}\left(
\mathbb{R}
^{N}\right)  }\frac{\left(  \int\limits_{%
\mathbb{R}
^{N}}\left\vert u\right\vert ^{r}\frac{dx}{\left\Vert x\right\Vert ^{s}%
}\right)  ^{1/r}}{\left(  \int\limits_{%
\mathbb{R}
^{N}}\left\Vert \nabla u\right\Vert _{\ast}^{p}\frac{dx}{\left\Vert
x\right\Vert ^{\mu}}\right)  ^{\frac{a}{p}}\left(  \int\limits_{%
\mathbb{R}
^{N}}\left\vert u\right\vert ^{q}\frac{dx}{\left\Vert x\right\Vert ^{\theta}%
}\right)  ^{\frac{1-a}{q}}}%
\]
is attained by some $\left\Vert \cdot\right\Vert -$radial optimizers.

\begin{proof}
Assume that $U_{0}$ is a $\left\Vert \cdot\right\Vert -$radial maximizer of
$CKN_{1}$. We set $V_{0}=T_{N,d,p}^{-1}U_{0}$. It means that $U_{0}%
=T_{N,d,p}V_{0}$. We note that $V_{0}$ is $\left\Vert \cdot\right\Vert
-$radial. Then for any $v$, we get
\[
\int\limits_{%
\mathbb{R}
^{N}}\left\vert v\right\vert ^{r}\frac{dx}{\left\Vert x\right\Vert ^{s}%
}=d^{1+\frac{p-1}{p}r}\int\limits_{%
\mathbb{R}
^{N}}\frac{\left\vert T_{N,d,p}v\right\vert ^{r}}{\left\Vert x\right\Vert
^{N+sd-Nd}}dx
\]%
\[
\int\limits_{%
\mathbb{R}
^{N}}\left\vert v\right\vert ^{q}\frac{dx}{\left\Vert x\right\Vert ^{\theta}%
}=d^{1+\frac{p-1}{p}q}\int\limits_{%
\mathbb{R}
^{N}}\frac{\left\vert T_{N,d,p}v\right\vert ^{q}}{\left\Vert x\right\Vert
^{N+\theta d-Nd}}dx
\]%
\[
\int\limits_{%
\mathbb{R}
^{N}}\frac{\left\Vert \nabla v\right\Vert _{\ast}^{p}}{\left\Vert x\right\Vert
^{\mu}}dx\geq\int\limits_{%
\mathbb{R}
^{N}}\left\Vert \nabla T_{N,d,p}v\right\Vert _{\ast}^{p}dx.
\]
and
\[
\int\limits_{%
\mathbb{R}
^{N}}\left\vert V_{0}\right\vert ^{r}\frac{dx}{\left\Vert x\right\Vert ^{s}%
}=d^{1+\frac{p-1}{p}r}\int\limits_{%
\mathbb{R}
^{N}}\frac{\left\vert U_{0}\right\vert ^{r}}{\left\Vert x\right\Vert
^{N+sd-Nd}}dx
\]%
\[
\int\limits_{%
\mathbb{R}
^{N}}\left\vert V_{0}\right\vert ^{q}\frac{dx}{\left\Vert x\right\Vert
^{\theta}}=d^{1+\frac{p-1}{p}q}\int\limits_{%
\mathbb{R}
^{N}}\frac{\left\vert U_{0}\right\vert ^{q}}{\left\Vert x\right\Vert
^{N+\theta d-Nd}}dx
\]%
\[
\int\limits_{%
\mathbb{R}
^{N}}\frac{\left\Vert \nabla V_{0}\right\Vert _{\ast}^{p}}{\left\Vert
x\right\Vert ^{\mu}}dx=\int\limits_{%
\mathbb{R}
^{N}}\left\Vert \nabla U_{0}\right\Vert _{\ast}^{p}dx.
\]
Hence%
\begin{align*}
&  \frac{\left(  \int\limits_{%
\mathbb{R}
^{N}}\left\vert v\right\vert ^{r}\frac{dx}{\left\Vert x\right\Vert ^{s}%
}\right)  ^{1/r}}{\left(  \int\limits_{%
\mathbb{R}
^{N}}\frac{\left\Vert \nabla v\right\Vert _{\ast}^{p}}{\left\Vert x\right\Vert
^{\mu}}dx\right)  ^{\frac{a}{p}}\left(  \int\limits_{%
\mathbb{R}
^{N}}\left\vert v\right\vert ^{q}\frac{dx}{\left\Vert x\right\Vert ^{\theta}%
}\right)  ^{\frac{1-a}{q}}}\leq\frac{\left(  d^{1+\frac{p-1}{p}r}\int\limits_{%
\mathbb{R}
^{N}}\frac{\left\vert T_{N,d,p}v\right\vert ^{r}}{\left\Vert x\right\Vert
^{N+sd-Nd}}dx\right)  ^{1/r}}{\left(  \int\limits_{%
\mathbb{R}
^{N}}\left\Vert \nabla T_{N,d,p}v\right\Vert _{\ast}^{p}dx\right)  ^{\frac
{a}{p}}\left(  d^{1+\frac{p-1}{p}q}\int\limits_{%
\mathbb{R}
^{N}}\frac{\left\vert T_{N,d,p}v\right\vert ^{q}}{\left\Vert x\right\Vert
^{N+\theta d-Nd}}dx\right)  ^{\frac{1-a}{q}}}\\
&  \leq d^{\frac{1}{r}+\frac{p-1}{p}-\frac{1-a}{q}-\frac{p-1}{p}\left(
1-a\right)  }\frac{\left(  \int\limits_{%
\mathbb{R}
^{N}}\frac{\left\vert U_{0}\right\vert ^{r}}{\left\Vert x\right\Vert
^{N+sd-Nd}}dx\right)  ^{1/r}}{\left(  \int\limits_{%
\mathbb{R}
^{N}}\left\Vert \nabla U_{0}\right\Vert _{\ast}^{p}dx\right)  ^{\frac{a}{p}%
}\left(  \int\limits_{%
\mathbb{R}
^{N}}\frac{\left\vert U_{0}\right\vert ^{q}}{\left\Vert x\right\Vert
^{N+\theta d-Nd}}dx\right)  ^{\frac{1-a}{q}}}\\
&  =\frac{\left(  \int\limits_{%
\mathbb{R}
^{N}}\left\vert V_{0}\right\vert ^{r}\frac{dx}{\left\Vert x\right\Vert ^{s}%
}\right)  ^{1/r}}{\left(  \int\limits_{%
\mathbb{R}
^{N}}\frac{\left\Vert \nabla V_{0}\right\Vert _{\ast}^{p}}{\left\Vert
x\right\Vert ^{\mu}}dx\right)  ^{\frac{a}{p}}\left(  \int\limits_{%
\mathbb{R}
^{N}}\left\vert V_{0}\right\vert ^{q}\frac{dx}{\left\Vert x\right\Vert
^{\theta}}\right)  ^{\frac{1-a}{q}}}.
\end{align*}
We note that we have the equality in the last row because $U_{0}$ and $V_{0}$
are $\left\Vert \cdot\right\Vert -$radial. Hence $V_{0}$ is a $\left\Vert
\cdot\right\Vert -$radial maximizer of $CKN_{2}.$
\end{proof}
\end{theorem}

As an application of Theorem \ref{T10}, to study the symmetry problem of
maximizers for CKN inequality (with the assumption that $\frac{N-p}{N-p-\mu
}>1$), we can assume that $\mu=0$.

\end{document}